\documentclass[reqno]{amsart}

\usepackage{amssymb}
\usepackage{amsthm}
\usepackage{amsfonts}
\usepackage{amsmath}
\usepackage{pmboxdraw}
\usepackage{verbatim} 
\usepackage{graphicx}
\usepackage{color}
\usepackage[colorlinks=true, citecolor=blue, filecolor=black, linkcolor=black, urlcolor=black]{hyperref}
\usepackage{cite}
\usepackage[normalem]{ulem}
\usepackage{subcaption}
\usepackage{bbm}
\usepackage{bm}
\usepackage{mathtools}
\usepackage{todonotes}
\usepackage{enumitem}
\usepackage{multirow,makecell}  
\usepackage{tablefootnote}
\usepackage{braket}

\usepackage[table]{xcolor}

\usepackage{xparse}
\usepackage{booktabs}
\usepackage{tcolorbox}
\tcbuselibrary{most}
\newtcolorbox{defn}[1]{breakable,
colbacktitle=gray!50!white, fonttitle=\bfseries, coltitle=black, title=Definition: {#1}}

\usepackage{tikz}
\usetikzlibrary{arrows.meta,positioning,calc}

\def\cN{\mathcal{N}}

\def\funcA{\mathfrak{A}}
\def\funcB{\mathfrak{B}}

\def\C{\mathbb{C}}
\def\D{\mathbb{D}}
\def\E{\mathbb{E}}
\def\HH{\mathbb{H}}

\def\R{\mathbb{R}}

\newcommand{\erfc}{\operatorname{erfc}}

\newcommand{\Tr}{\operatorname{Tr}}
\newcommand{\bfA}{\mathbf{A}}
\newcommand{\bfB}{\mathbf{B}}

\newcommand{\bfE}{\mathbf{E}}

\newcommand{\bfL}{\mathbf{L}}
\newcommand{\bfM}{\mathbf{M}}
\newcommand{\bfQ}{\mathbf{Q}}

\newcommand{\Li}{\operatorname{Li}}
\newcommand{\sgn}{\operatorname{sgn}}
\newcommand{\Var}{\operatorname{Var}}
\newcommand{\Cov}{\operatorname{Cov}}

\newcommand{\im}{\operatorname{Im}} 

\theoremstyle{plain}
\newtheorem*{thm*}{Theorem}
\newtheorem{thm}{Theorem}[section]
\newtheorem{lem}[thm]{Lemma}
\newtheorem{cor}[thm]{Corollary}
\newtheorem{prop}[thm]{Proposition}
\newtheorem*{prop*}{Proposition}
\newtheorem*{lem*}{Lemma}

\theoremstyle{definition}
\newtheorem{eg*}{Example}
\newtheorem{egs*}{Examples}
\newtheorem*{Q*}{Question}

\theoremstyle{remark} 
\newtheorem{rem}{Remark}
\newtheorem*{rmk*}{Remark}
\newtheorem*{rmks*}{Remarks}

\numberwithin{equation}{section}

\begin{document}
\title{Disc counting statistics of the real Ginibre ensemble}  

\author{Sung-Soo Byun}
\address{Department of Mathematical Sciences and Research Institute of Mathematics, Seoul National University, Seoul 151-747, Republic of Korea}
\email{sungsoobyun@snu.ac.kr}

\author{Yong-Woo Lee}
\address{Department of Mathematical Sciences, Seoul National University, Seoul 151-747, Republic of Korea}
\email{hellowoo@snu.ac.kr}




\begin{abstract}
We study the disc counting statistics of the real Ginibre ensemble, whose spectrum consists of real and non-real eigenvalues, the latter occurring in complex-conjugate pairs. As the matrix dimension increases, we derive the asymptotic behaviour of the joint cumulants of the numbers of real and non-real eigenvalues contained in a centred disc in the bulk, edge, and exterior regimes of the circular law. As consequences, we establish a joint central limit theorem and obtain the conjectured asymptotics for the number variance. Furthermore, our results reveal a universality phenomenon: although the fluctuations of the real and non-real eigenvalues separately retain dependence on the symmetry class, their combined fluctuations exhibit the same limiting behaviour as in the previously established complex and symplectic Ginibre ensembles. 
\end{abstract}

\maketitle

\section{Introduction}

Counting statistics is one of the most fundamental observables associated with random point processes. Given a domain $D$, it records the number of particles contained in $D$, thereby quantifying global fluctuations beyond those encoded by local correlation functions. Understanding the asymptotic behaviour of its variance, higher-order cumulants, and limiting distributions as the system size tends to infinity has therefore become a central theme in the study of point processes. 
In particular, the growth of the number variance provides a signature of important statistical properties such as hyperuniformity and rigidity, while central limit theorems reveal universal fluctuation phenomena in strongly correlated particle systems. As a result, counting statistics has been studied extensively for a wide variety of models, including determinantal point processes, Coulomb gases, and other interacting particle systems.

Counting statistics has also become a central topic in non-Hermitian random matrix theory, where the Ginibre ensembles provide canonical examples; see \cite{BF25} for a review.  
They comprise three symmetry classes: the real, complex, and symplectic Ginibre ensembles, denoted by GinOE, GinUE, and GinSE, respectively. All three ensembles possess a rich algebraic structure: GinUE forms a determinantal point process, whereas GinOE and GinSE form Pfaffian point processes.

From a statistical physics point of view, the GinUE admits an interpretation as a rotationally invariant planar Coulomb gas \cite{Fo10,BF25a}, whereas the GinSE is naturally associated with a Coulomb gas having complex-conjugation symmetry. Although the GinSE lacks rotational symmetry, its Pfaffian structure nevertheless allows the angular variables to be integrated out explicitly; see e.g. \cite{Rider03,AIS14}. Owing to these features, radial counting observables for the GinUE and GinSE are remarkably tractable and have consequently been studied extensively. Counting statistics for these ensembles has also attracted considerable attention in the physics literature, particularly in connection with rotating free fermions; see e.g. \cite{LGCCKMS19,LGMS18,LMS19,MMS14,SLMS21,SLMS22,GDMK26} and references therein. More recently, Charlier obtained remarkably precise asymptotic expansions for the full counting statistics \cite{Ch22} as well large deviation probabilities \cite{Ch23} (see too the earlier work \cite{CMV16}), thereby initiating a series of subsequent works establishing similar results for a variety of related models within these symmetry classes \cite{ACCL23,ACCL24,ACM24,CL23,ABE23,ABES23,BP26,Noda25,AFLS25,BC25}.

In contrast, the GinOE presents a fundamentally different feature. Unlike the GinUE and GinSE, whose spectra consist entirely of non-real eigenvalues, the spectrum of the GinOE contains both real eigenvalues and pairs of non-real complex conjugate eigenvalues. Consequently, the spectrum no longer forms a single-species point process but instead consists of two interacting components supported on different dimensions. This mixed-dimensional structure fundamentally complicates the analysis and prevents the radial decomposition that underlies the study of counting statistics in the GinUE and GinSE. As a consequence, even the fundamental problem of counting eigenvalues inside a disc becomes substantially more challenging and remains largely open.

This intrinsic difficulty is reflected in the existing literature. Existing results on the GinOE have almost exclusively focused on \textit{one component of the spectrum} at a time, namely either the real eigenvalues or the non-real eigenvalues. Real eigenvalue statistics of the GinOE also arise naturally in applications, notably in the counting and stability analysis of equilibria in high-dimensional random dynamical systems; see e.g. \cite{BFK21,Fyo16,FK16,May72,Kiv24}. For the real eigenvalues, counting statistics have been studied extensively, including law of large numbers \cite{EKS94,FN07}, central limit theorems \cite{Si17,FS23a,Fo24}, and large deviation probabilities \cite{KPTTZ15,ABL25,AK07,MPTW16}.  For the non-real eigenvalues, the Gaussian  fluctuation was established recently in \cite{GLX24}. Such results, predominantly concerning the real eigenvalues, have also been established for a variety of other observables \cite{CESX22,Fo15a,XZ24,BLO26,BJLS25,ATW14,RS14,LM26}, as well as for various
generalisations of the GinOE \cite{AB23,BKLL23,BN25,Fo14,Fo25,FI16,FIK20,FK18,FM12,FN08,GP19,Si17a,BL24,LMS22,TV15}. 

While these works provide a detailed understanding of the fluctuations of each component individually, they do not describe the fluctuations of the full spectrum. Indeed, disc counting statistics depend fundamentally on the \emph{joint fluctuations} of the real and non-real eigenvalues and therefore cannot be recovered from separate analyses of the two components. Results treating both parts of the spectrum simultaneously remain extremely scarce. One notable exception is a recent work \cite{CEK26}, which established hyperuniformity for general real i.i.d.\ random matrices by showing that the number variance grows polynomially more slowly than its expectation.

The main purpose of the present paper is to provide a comprehensive study of disc counting statistics for the GinOE. In Theorem~\ref{thm: cumulants}, we derive explicit asymptotic formulas for the \textit{joint cumulants} of the numbers of real and non-real eigenvalues contained in a disc in all natural macroscopic regimes, namely in the bulk of the circular law, near the spectral edge, and outside the limiting droplet; see Figure~\ref{fig:examples}. As a consequence, in Corollary~\ref{cor: CLT}, we establish a joint central limit theorem for the two counting functions and determine their covariance structure. In particular, we obtain the asymptotic number variance, thereby confirming the conjectures proposed in a recent work \cite[Conjecture~2.10]{ABES23}; cf. Table~\ref{tab:main-statements}.
It is worth noting that while the number of non-real eigenvalues is typically of order $N$ and that of the real eigenvalues is only of order $\sqrt{N}$, the variances of both counting functions, as well as their covariance, are all of order $\sqrt{N}$; cf. Proposition~\ref{prop: number variance}. 

Beyond resolving the conjectures on the number variance asymptotics, our results show a new universality phenomenon. Considered separately, the real and non-real eigenvalue counting functions each exhibit symmetry-dependent fluctuations that differ substantially from those of the GinUE and GinSE. Remarkably, these symmetry-dependent contributions
cancel exactly when the two species are combined. Consequently, the
limiting cumulants of the disc counting function for the entire spectrum
coincide with those of the GinUE and GinSE in all macroscopic regimes.
Thus, although universality across different symmetry classes is lost at the level of individual species, it is restored for the full spectrum through the joint fluctuations of the real and non-real eigenvalues. We refer the reader to Remark~\ref{Rem_universality} for a more detailed discussion.

\begin{figure}[t]
    \begin{subfigure}{0.32\textwidth}
        \begin{center}
            \includegraphics[width=\linewidth]{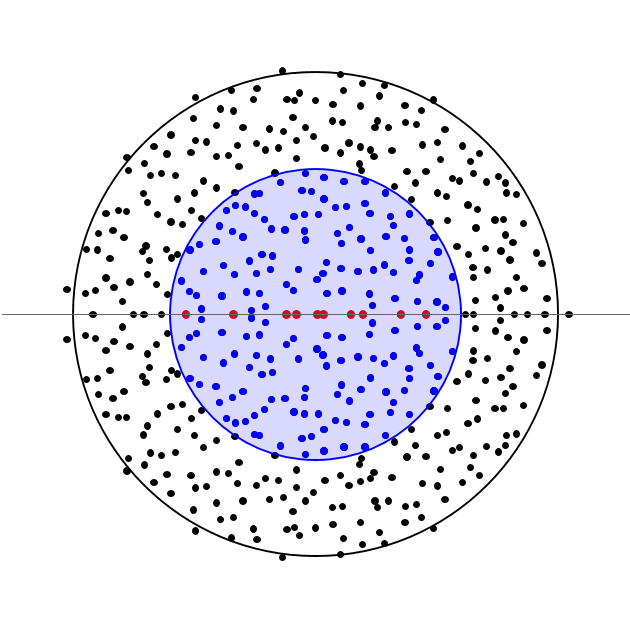}
        \end{center}
        \subcaption{ Bulk ($R=0.6$) }  
    \end{subfigure}
    \begin{subfigure}{0.32\textwidth}
        \begin{center}
            \includegraphics[width=\linewidth]{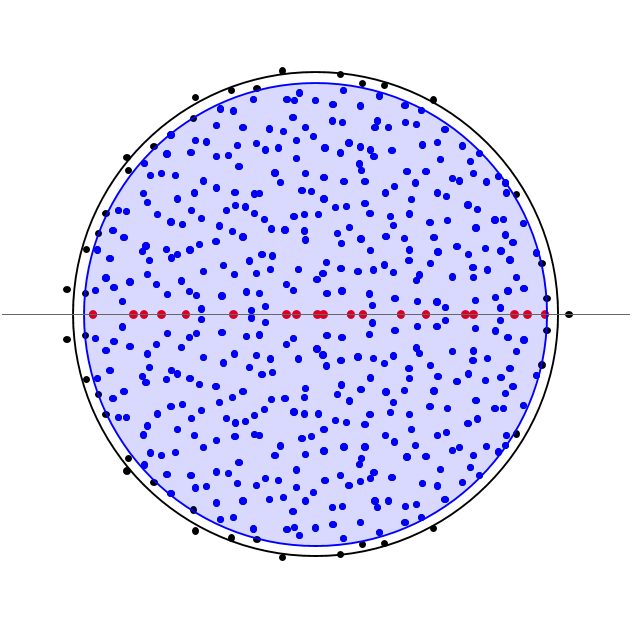}
        \end{center}
        \subcaption{ Edge ($R=1-1/\sqrt{N}$)  } 
    \end{subfigure}
    \begin{subfigure}{0.32\textwidth}
        \begin{center}
            \includegraphics[width=\linewidth]{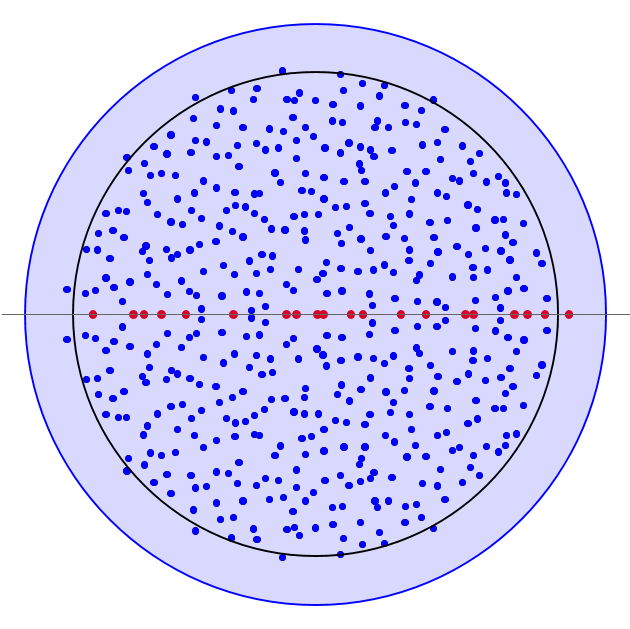}
        \end{center}
        \subcaption{ Outside ($R=1.2$)  }
    \end{subfigure} 
    \caption{ 
    Three different regimes for disc counting statistics of radius $R$. The figures used a single realisation of the GinOE of size $N=500$. The real eigenvalues and complex conjugate pairs inside the discs are coloured red and blue, respectively. 
    }
    \label{fig:examples}
\end{figure}

\begin{table}[t]
\centering
\renewcommand{\arraystretch}{1.3}
\setlength{\tabcolsep}{10pt}
\begin{tabular}{lccc}
\toprule
\rowcolor{gray!12}
Observable
& Variance
& CLT
& Higher cumulants
\\
\midrule
Real
& \cite{FN07,FS23a}
& \cite{FS23a,Fo24}
& Theorem~\ref{thm: cumulants}
\\[0.2em]

Non-real
& \cite{GLX24}
& \cite{GLX24}
& Theorem~\ref{thm: cumulants}
\\[0.2em]

Full spectrum
& Proposition~\ref{prop: number variance}
& Corollary~\ref{cor: CLT}
& Theorem~\ref{thm: cumulants}
\\
\bottomrule
\end{tabular}
\caption{Summary of known and present results on counting statistics for the GinOE.}
\label{tab:main-statements}
\end{table}

\section{Main results}

Let us now state our results more precisely. The GinOE is the ensemble of $N\times N$ random matrices whose entries are independent centred real Gaussian random variables with variance $1/N$. It is well known that, as $N\to\infty$, the empirical spectral measure converges weakly to the uniform probability measure on the unit disc centred at the origin, known as the circular law; see e.g. \cite{Ed97}. 

To describe the disc counting statistics, let $R \in (0,\infty]$ and consider the centred disc of radius $R$. We denote by $\cN_N^{\mathrm r}(R)$ and $\cN_N^{\mathrm c}(R)$ the numbers of real and non-real eigenvalues of the GinOE contained in this disc, respectively. (Here, the superscripts ``$\mathrm r$'' and ``$\mathrm c$'' stand for \emph{real} and \emph{complex}; although the latter counts only the non-real eigenvalues, we adopt this notation for simplicity.) Their sum,
\begin{equation}
\cN_N(R):=\cN_N^{\mathrm r}(R)+\cN_N^{\mathrm c}(R),
\end{equation}
is the total number of eigenvalues contained in the disc. Note in particular that 
\begin{equation} \label{total variance formula}
\Var \cN_N(R)= \Var \cN_N^{\mathrm r}(R) + \Var \cN_N^{\mathrm c}(R) + 2 \Cov \big( \cN_N^{\mathrm r}(R),\cN_N^{\mathrm c}(R) \big).  
\end{equation}

As previously mentioned, extensive work has been devoted to the counting statistics of the real eigenvalues. It follows from \cite{EKS94} that for every (fixed) $R>0$, 
\begin{equation} \label{expected number real evs_EKS}
\lim_{N \to \infty} \frac{ \mathbb{E} \cN_N^{ \rm r }(R) }{ \sqrt{2N/\pi} } =  \min\{R,1\}. 
\end{equation}
Moreover, it was established in \cite{FN07,FS23a} that 
\begin{equation} \label{number variance real evs_FS}
 \lim_{N \to \infty} \frac{ \Var \cN_N^{ \rm r }(R) }{ \sqrt{2N/\pi} } = (2-\sqrt{2})  \min\{R,1\}. 
\end{equation} 
By \eqref{expected number real evs_EKS} and \eqref{number variance real evs_FS}, both the expected number of real eigenvalues and its variance are of order $\sqrt{N}$.  

In contrast, much less is known about the counting statistics of the non-real eigenvalues. The expected number of non-real eigenvalues follows from \cite{EKS94} and, at leading order, is simply proportional to the area of the counting domain, as predicted by the convergence to the circular law. In particular,
\begin{equation} \label{expected number complex evs_EKS}
    \lim_{N\to\infty}
    \frac{\mathbb{E}\cN_N^{\mathrm c}(R)}{N}
    =
    \min\{R^2,1\};
\end{equation}
see also \cite[Proposition~2.1]{ABES23}. 
The fluctuations, however, are considerably more subtle. The only available variance asymptotics are due to \cite{GLX24}, where the counting domain is assumed to be contained in the unit disc and strictly separated from the real axis. More precisely, for any Jordan domain $D\subset\{|z|<1\}$ whose closure is bounded away from the real axis, it was shown that 
\begin{equation} \label{number variacne length GLX24}
    \Var\,( \textup{\# of eigenvalues in $D$} )
    \sim
    \sqrt{N}\,
    \frac{\ell(\partial D)}{2\pi\sqrt{\pi}}, \qquad \textup{if } D \cap \R = \varnothing, 
\end{equation}
where $\ell(\partial D)$ denotes the perimeter of $D$. On the other hand, the variance of the number of non-real eigenvalues in domains intersecting the real axis has remained unknown. 

Furthermore, the covariance between the numbers of real and non-real eigenvalues has remained largely unexplored. While it is natural to expect these quantities to be negatively correlated, the asymptotic behaviour of their covariance has, to the best of our knowledge, never been quantified rigorously. 
Our first result fills this gap by providing a complete description of the asymptotic covariance structure, including the number variances of the real, non-real, and total eigenvalue counts.

We mention that as in most of the existing literature, we restrict our attention to the case where the matrix dimension $N$ is even, as this simplifies the presentation; see, however, \cite{FM09,Si07} for treatments of the odd-dimensional case.

\begin{prop}[\textbf{Number variance and covariance asymptotics}] \label{prop: number variance}
    Let $N$ be a positive even integer, and $R\equiv R_N$ be a positive real number.
    \begin{enumerate}[label=\textup{(\roman*)}]
        \item \textup{\textbf{(Bulk)}}
        Suppose that $R\in(0,1)$ is fixed. Then 
        \begin{equation}
            \lim_{N\to\infty}\frac{ \Var \cN_N^{ \rm c}(R) }{ \sqrt{2N/\pi} } = 2R , \qquad  \lim_{N\to\infty}\frac{ \Var \cN_N^{\rm r}(R) }{ \sqrt{2N/\pi} } = (2-\sqrt{2}) R,
        \end{equation}
        and
        \begin{equation}
            \lim_{N\to\infty}  \dfrac{\Cov( \cN_N^{ \rm r }(R), \cN_N^{ \rm c }(R) )  }{  \sqrt{2N/\pi} } = -(2-\sqrt{2}) R.
        \end{equation}
        Consequently, we have
        \begin{equation} \label{number variance_total bulk}
            \lim_{N\to\infty}\frac{ \Var \cN_N(R) }{ \sqrt{2N/\pi}  } = \sqrt{2}R.
        \end{equation}
        
        \item \textup{\textbf{(Edge)}} 
        Suppose that $R = 1- s/\sqrt{2N}$ for some $s\in\R$. Then 
        \begin{equation}
            \lim_{N\to\infty}\frac{ \Var \cN_N^{ \rm c }(R) }{ \sqrt{2N/\pi} } = \sqrt{2}\,\mathcal{F}(s) + 2-\sqrt{2}, \qquad  \lim_{N\to\infty}\frac{ \Var \cN_N^{ \rm r }(R) }{ \sqrt{2N/\pi} } = 2-\sqrt{2},
        \end{equation}
        and
        \begin{equation}
             \lim_{N\to\infty} \dfrac{  \Cov( \cN_N^{ \rm r }(R), \cN_N^{ \rm c }(R) )  }{  \sqrt{2N/\pi} } = -(2-\sqrt{2}),
        \end{equation}
        where
        \begin{equation} \label{def of f number variance}
            \mathcal{F}(s)  :=  \sqrt{ 2\pi }   \int_{-\infty}^{s} \frac{\erfc(t)\erfc(-t)}{4}  \,dt.
        \end{equation}
        Consequently, we have
        \begin{equation} \label{number variance_total edge}
            \lim_{N\to\infty} \frac{\Var \cN_N(R)}{\sqrt{2N/\pi}} = \sqrt{2} \, 
            \mathcal{F}(s).
        \end{equation} 
        
        \item \textup{\textbf{(Outside)}}
        Suppose that $R \in (1,\infty]$ is fixed. Then 
        \begin{equation}
            \lim_{N\to\infty}\frac{\Var \cN_N^{ \rm c }(R) }{\sqrt{2N/\pi}} = 2-\sqrt{2}, \qquad \lim_{N\to\infty}\frac{\Var \cN_N^{\rm r  }(R) }{\sqrt{2N/\pi}} = 2 - \sqrt{2},
        \end{equation}
        and
        \begin{equation}
            \lim_{N\to\infty} \frac{ \Cov( \cN_N^{ \rm r }(R), \cN_N^{ \rm c }(R) ) }{ \sqrt{2N/\pi} } = -(2-\sqrt{2}).
        \end{equation}
        Consequently, we have
        \begin{equation} \label{number variance_total outside}
            \lim_{N\to\infty} \frac{\Var \cN_N(R) }{\sqrt{2N/\pi}} = 0.
        \end{equation}
    \end{enumerate}
\end{prop}

See Figure~\ref{fig: Numerics} for the numerical illustrations of Proposition~\ref{prop: number variance}.

\begin{figure}[t]
    \begin{subfigure}{0.45\textwidth}
        \begin{center}
            \includegraphics[width=\linewidth]{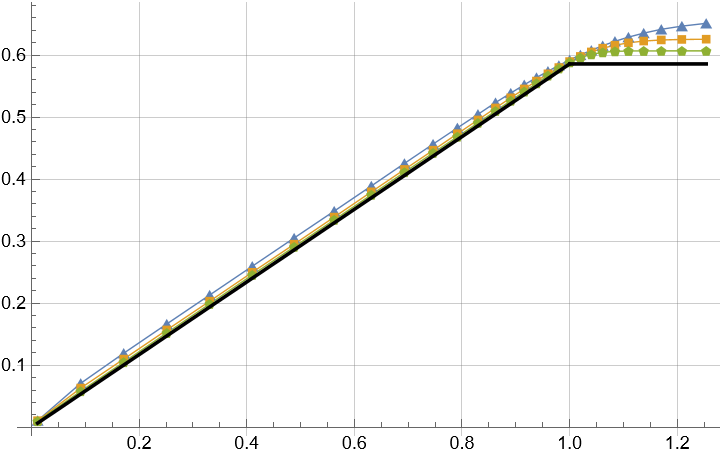}
        \end{center}
        \subcaption{$\Var \cN_N^{ \rm r }(R) / \sqrt{2N/\pi}$}  \label{fig:real_var}
    \end{subfigure}
    \quad
    \begin{subfigure}{0.45\textwidth}
        \begin{center}
            \includegraphics[width=\linewidth]{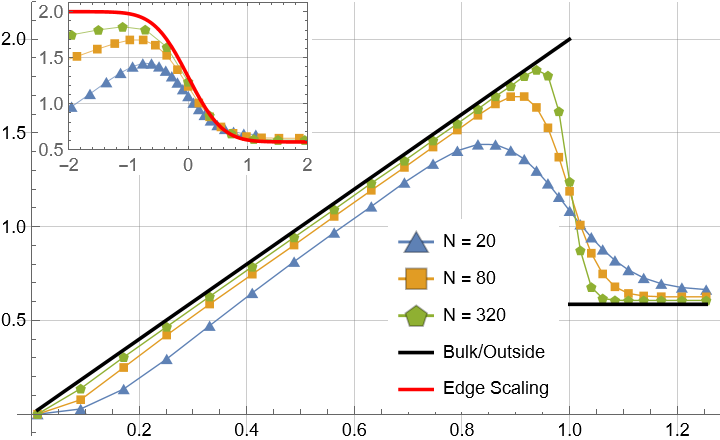}
        \end{center}
        \subcaption{$\Var \cN_N^{ \rm c }(R)/ \sqrt{2N/\pi}$}  \label{fig:comp_var}
    \end{subfigure}

    \vspace{1em}
    
    \begin{subfigure}{0.45\textwidth}
        \begin{center}
            \includegraphics[width=\linewidth]{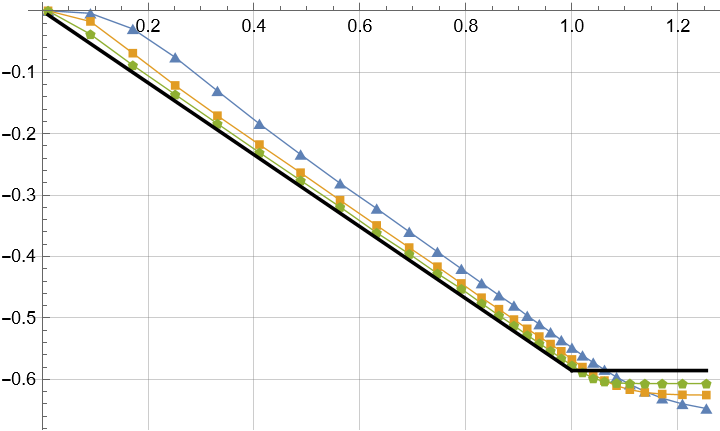}
        \end{center}
        \subcaption{$\Cov(\cN_N^{ \rm r }(R),\cN_N^{ \rm c }(R))/ \sqrt{2N/\pi}$}  \label{fig:cov}
    \end{subfigure}
    \caption{Numerical plots for the joint cumulants of $\cN_N^{ \rm c }(R)$ and $\cN_N^{ \rm r }(R)$. The inset in (B) depicts the edge scaling regimes. The cumulants are computed from $10^6$ independent samples of the GinOEs.
    }
    \label{fig: Numerics}
\end{figure}

As mentioned above, the asymptotic number variance for the real eigenvalues in the bulk and exterior regimes was previously established in \cite{FS23a,FN07}; see \eqref{number variance real evs_FS}. On the one hand, Proposition~\ref{prop: number variance} confirms \cite[Conjecture~2.10]{ABES23}. In fact, our results provide more information beyond this conjecture. While \cite{ABES23} proposed conjectural formulas for the covariance asymptotics in the bulk and exterior regimes, no conjecture was made for the edge regime. Proposition~\ref{prop: number variance} establishes the missing edge asymptotics and thereby completes the picture for all macroscopic scaling regimes. 

\begin{rem}[Heuristic explanation of the variance growth]
Proposition~\ref{prop: number variance} shows that the number variances of both the real and non-real eigenvalues are of order $\sqrt{N}$. These two statistics nevertheless exhibit rather different behaviour from the viewpoint of hyperuniformity, since the typical numbers of real and non-real eigenvalues are of orders $\sqrt{N}$ and $N$, as in \eqref{expected number real evs_EKS} and \eqref{expected number complex evs_EKS}, respectively.

There is, however, a natural heuristic explanation for the common $\sqrt{N}$-order scaling of their variances. From the perspective of statistical physics, number fluctuations are expected to be generated primarily through the boundary of the counting domain, leading to a variance proportional to its boundary length. This principle is also reflected in the asymptotic formula \eqref{number variacne length GLX24} for domains bounded away from the real axis.

For centred discs, the geometry is fundamentally different. The non-real eigenvalues form a two-dimensional point process, for which fluctuations are expected to be generated by the boundary of the counting domain. In the present setting, this boundary consists of the semicircular arc together with the diameter $(-R,R)$ lying on the real axis. Thus, in addition to the contribution from the semicircular boundary, the fluctuations of the non-real eigenvalues may contain a contribution associated with this real-axis segment. On the other hand, the real eigenvalues form a one-dimensional point process on
$\mathbb{R}$, and their counting statistic in the centred disc is precisely supported on the same interval $(-R,R)$. Hence the interval $(-R,R)$ plays a distinguished role in both statistics: it is a boundary component of the two-dimensional counting domain for the non-real eigenvalues and the counting domain itself for the real eigenvalues. This common one-dimensional geometry makes it natural that the corresponding number variances are of the same order. \end{rem}

\begin{rem}[Comparison with the number variance of the GinUE]
We note that the number variance of the GinUE, or more generally of
determinantal Coulomb gases (random normal matrix ensembles), has recently
been studied in \cite{FL22,MMO26,MAD24}. In particular, for the GinUE, it
was shown that the number variance asymptotic is precisely of the form
\eqref{number variacne length GLX24}. Intuitively, this is natural, since away from the real axis, several spectral statistics of the non-real eigenvalues of the GinOE are expected to agree, to leading order, with their GinUE counterparts; see, e.g. \cite{BS09,LO25,GLX24,BCD26,CL26,CESX22}.  

On the other hand, Proposition~\ref{prop: number variance} reveals an
intriguing feature of the number variance for the total eigenvalue count.
Indeed, \eqref{number variance_total bulk} implies that if $D$ is the
centred disc of radius $R\in(0,1)$, then
\begin{equation}\label{number variacne length R intersect}
    \Var \bigl(\#\{\textup{eigenvalues in }D\}\bigr)
    \sim
    \sqrt{N}\,
    \frac{\ell(\partial D)}{\pi\sqrt{\pi}},
    \qquad
    D\cap\mathbb R\neq\varnothing.
\end{equation}
This should be compared with the asymptotic formula
\eqref{number variacne length GLX24} for the non-real eigenvalues, where
the counting domain is assumed to be bounded away from the real axis.
The two formulas differ by a factor of $2$.

Let us give a heuristic interpretation of this factor. Write $ D_+ := D\cap\mathbb C_+.$  Away from the real axis, the non-real eigenvalues in $D_+$ are expected
to exhibit, to leading order, the same local fluctuations as the GinUE.
There is, however, an important subtlety in applying the usual boundary
law directly to $D_+$. Its boundary consists of the upper semicircular
arc together with the interval $D\cap\mathbb R$. The latter lies precisely
at the interface with the real eigenvalues, and therefore cannot be
regarded as an ordinary bulk boundary for the non-real eigenvalues.
For the total eigenvalue count, the contribution associated with this
real-axis boundary is cancelled by the fluctuations involving the real
eigenvalues, including their covariance with the non-real eigenvalue
count. Thus, heuristically, only the upper semicircular part of the
boundary contributes to the leading-order fluctuation.

Since this semicircular arc has length $\ell(\partial D)/2$, the GinUE
boundary law suggests
\[
    \Var\bigl(\#\{\textup{non-real eigenvalues in }D_+\}\bigr)
    \sim
    \sqrt{N}\,
    \frac{\ell(\partial D)}{4\pi\sqrt{\pi}},
\]
after the real-axis contribution has been cancelled as described above.
Finally, the non-real eigenvalues of the GinOE occur in complex-conjugate
pairs. Hence the fluctuations of the non-real eigenvalue counts in the
upper and lower half-discs are perfectly correlated: each eigenvalue in
$D_+$ contributes together with its conjugate in $D_-$. Consequently,
the corresponding fluctuation in the full disc is doubled, and its
variance is therefore multiplied by $4$. This gives
\[
    4\,
    \frac{\sqrt{N}\,\ell(\partial D)}
         {4\pi\sqrt{\pi}}
    =
    \sqrt{N}\,
    \frac{\ell(\partial D)}{\pi\sqrt{\pi}},
\]
in agreement with \eqref{number variacne length R intersect}.

This heuristic suggests that the additional factor $2$ compared with
the GinUE is a manifestation of the complex-conjugate pairing of the
GinOE spectrum, while the effect of the real axis itself is cancelled
in the total eigenvalue count through the joint fluctuations of the
real and non-real eigenvalues; cf. \cite[Remark 2.9]{ABES23}.
\end{rem}

\begin{rem}[Cutoff of the number variance]
Observe from \eqref{number variance_total bulk} and \eqref{number variance_total outside} that the number variance of the total eigenvalue counting function exhibits a sharp cutoff once the disc extends beyond the support of the circular law. This transition is encoded by the scaling function $\mathcal{F}$ in \eqref{def of f number variance}, which describes the edge crossover.

Our asymptotic formulas for the real and non-real eigenvalue counting
functions further reveal that this cutoff is entirely driven by the
variance of the non-real eigenvalues; see Figure~\ref{fig: Numerics} (B). Indeed, while the contribution of
the real eigenvalues changes smoothly across the edge, the non-real
eigenvalue variance undergoes the crossover described by $\mathcal{F}$. This reflects the fact that the boundary of the circular-law droplet influences the non-real eigenvalues at the macroscopic scale.  
\end{rem}

Next, we turn to our main result on the asymptotic behaviour of the joint cumulants.

To state the result, we first recall the definition of joint cumulants. Given two random variables $X$ and $Y$, we denote by $\kappa_{j,k}(X,Y)$ their $(j,k)$-th joint cumulant, defined by
\begin{equation}\label{def of joint cumulant in general}
\kappa_{j,k}(X,Y)
:=
\left.
\frac{\partial^{\,j+k}}
{\partial t^j\,\partial u^k}
\log
\mathbb{E}
\bigl[
e^{\,tX+uY}
\bigr]
\right|_{t=u=0},
\qquad j,k\ge0.
\end{equation}
In particular, when $j+k=2$, the joint cumulants recover the variance and covariance via
\begin{equation} \label{joint cumulants special variances}
\kappa_{2,0}(X,Y)=\Var(X), \qquad
\kappa_{0,2}(X,Y)=\Var(Y), \qquad
\kappa_{1,1}(X,Y)=\Cov(X,Y).
\end{equation}  
We also write $\kappa_p(X)$ for the $p$-th cumulant of a random variable $X$. 

In order to state the asymptotic formula for the joint cumulants, we introduce the following quantities. For a positive integer $q$ and $s\in\mathbb{R}\cup\{+\infty\}$, define
\begin{equation} \label{def of varkappa real}
 \varkappa_q^{ \rm r } := 2^{q-1} \sum_{m = 1}^q \frac{(-1)^{m+1} m!}{m^{3/2}} S(q,m),
\end{equation} 
where $S(q,m)$ denotes the Stirling number of the second kind (see e.g. \cite[Section 26.8]{NIST})
\begin{equation}
S(q,m)
=
\frac1{m!}
\sum_{k=0}^{m}
(-1)^{m-k}
\binom{m}{k}
k^q. 
\end{equation}
In particular, the first few values of $\varkappa_q^{\rm r}$ are given by 
\begin{equation} \label{varkappa special real}
\varkappa_1^{\rm r}=1,\qquad
\varkappa_2^{\rm r}=2-\sqrt2,\qquad
\varkappa_3^{\rm r}
=
4-6\sqrt2+\frac{8 \sqrt{3} }{3}.
\end{equation}
For $q \ge 2$, we also define 
\begin{equation} \label{def of varkappa total}
  \varkappa_q(s)  :=(-2)^{q-1} \sqrt{\pi} \int_{-\infty}^s \Li_{1-q} \bigg(-\frac{\erfc(-x)}{\erfc(x)} \bigg) \, dx, 
\end{equation}
where $\Li_\nu(z)$ is the polylogarithm function (see e.g. \cite[Section 25.12]{NIST}) 
\begin{equation} \label{def of polylog}
\Li_\nu(z)
=
\sum_{k=1}^{\infty}
\frac{z^k}{k^\nu},
\qquad |z|<1,
\end{equation}
defined elsewhere by analytic continuation. Here, $\erfc(x)$  is the complementary error function \cite[Chapter~7]{NIST}.
Notice here that by definition \eqref{def of polylog} we have $\Li_{-1}(z)=z/(1-z)^2.$ 
Using this, one can observe that 
\begin{equation} \label{varkappa special total}
\varkappa_2(s)
=
\sqrt2\,\mathcal{F}(s), 
\end{equation}
where $\mathcal{F}$ is given by \eqref{def of f number variance}. We also remark that for odd integers $q\ge3$, due to the symmetry, $ \varkappa_q(\infty)=0$; cf. \cite[Eq.~(2.86)]{ABES23}.
 
Our main result provides asymptotic formulas for the joint cumulants of arbitrary orders of the numbers of real and non-real eigenvalues, thereby generalising Proposition~\ref{prop: number variance}. 

\begin{thm}[\textbf{Asymptotics of the joint cumulants}] \label{thm: cumulants}
    Let $N$ be a positive even integer and let $R\equiv R_N$ be a positive real number.
    For fixed nonnegative integers $j$ and $k$ satisfying $j+k\ge2$, let $\varkappa_k^{\rm r}$ and $\varkappa_k$ be as in \eqref{def of varkappa real} and \eqref{def of varkappa total}. Then the following asymptotic formulas hold.
    \begin{enumerate}[label=\textup{(\roman*)}]
        \item \textup{\textbf{(Bulk)}} Suppose that $R\in(0,1)$ is fixed. Then 
        \begin{align}
            \lim_{N\to\infty} \frac{\kappa_{j,k}(\cN_{N}^{ \rm r }(R), \cN_N^{ \rm c }(R))}{R \sqrt{2N/\pi}} = \begin{dcases}
                (-1)^k \varkappa_{j+k}^{ \rm r } , &  \textup{if }  j\neq 0,
                \smallskip 
                \\
                (-1)^k \varkappa_k^{ \rm r }  + \varkappa_k(\infty),  &  \textup{if } j = 0. 
            \end{dcases}  \label{joint cumulants bulk main thm}
        \end{align} 
         In particular, for every integer $p\ge2$,
        \begin{equation} \label{cumulants_total bulk}
        \lim_{N\to\infty}
        \frac{\kappa_p\bigl(\mathcal N_N(R)\bigr)}
        {R\sqrt{2N/\pi}}
        = \varkappa_p(\infty), 
        \end{equation}
        where $\varkappa_p(\infty)=0$ whenever $p$ is odd. 
        
        \smallskip 
        
        \item \textup{\textbf{(Edge)}}  Suppose that $R = 1- s/\sqrt{2N}$ for some $s\in\R$. Then  
        \begin{equation}    \label{joint cumulants edge main thm}
            \lim_{N\to\infty} \frac{ \kappa_{j,k}(\cN_{N}^{ \rm r }(R), \cN_N^{ \rm c }(R))  }{\sqrt{2N/\pi}} = \begin{dcases}
                (-1)^k \varkappa_{j+k}^{ \rm r}, &  \textup{if } j\neq 0,
              \smallskip 
              \\
                (-1)^k \varkappa_k^{ \rm r} + \varkappa_k(s),  &  \textup{if } j = 0.
            \end{dcases}
        \end{equation}
        In particular, for every integer $p\ge2$,
        \begin{equation}  \label{cumulants_total edge}
        \lim_{N\to\infty} 
        \frac{\kappa_p\bigl(\mathcal N_N(R)\bigr)}
        {\sqrt{2N/\pi}}
        =
        \varkappa_p(s).
        \end{equation}
        
        \smallskip 
        
        \item \textup{\textbf{(Outside)}} Suppose that $R \in (1,\infty]$ is fixed. Then 
        \begin{equation}  \label{joint cumulants outside main thm}  
            \lim_{N\to\infty} \frac{\kappa_{j,k}(\cN_{N}^{ \rm r }(R), \cN_{N}^{ \rm c }(R) )}{\sqrt{2N/\pi}} = (-1)^k \varkappa_{j+k}^{ \rm r }.
        \end{equation}
        In particular, for every integer $p\ge2$,
        \begin{equation} \label{cumulants_total outside}
        \lim_{N\to\infty} 
        \frac{\kappa_p\bigl(\mathcal N_N(R)\bigr)}
        {\sqrt{2N/\pi}}
        = 0.
        \end{equation}
    \end{enumerate} 
\end{thm}

By \eqref{joint cumulants special variances},
\eqref{varkappa special real}, and
\eqref{varkappa special total}, the case $j+k=2$ of
Theorem~\ref{thm: cumulants} immediately recovers
Proposition~\ref{prop: number variance}. 
Nevertheless, we state Proposition~\ref{prop: number variance} separately, as the second-order case deserves special attention for being more accessible from an intuitive point of view.

\begin{rem}[Universality across symmetry classes] \label{Rem_universality}
A remarkable consequence of Theorem~\ref{thm: cumulants} is a universality phenomenon for the asymptotic cumulants of the total eigenvalue counting function; see
\eqref{cumulants_total bulk},
\eqref{cumulants_total edge}, and
\eqref{cumulants_total outside}.
Indeed, under the appropriate symmetry-dependent normalisations, the same limiting cumulants $\varkappa_p$ arise in the
corresponding counting statistics for the GinUE and GinSE; see
Eq.~(10) of the arXiv version of \cite{LMS19} for the GinUE and
\cite[Theorem~2.14]{ABES23} for the GinSE.

What makes this universality particularly interesting is that it is invisible at the level of the individual species.
The cumulants of the real and non-real eigenvalue counting functions
depend strongly on the underlying symmetry class and differ
substantially from those of the GinUE and GinSE.
On the other hand, when the two counting functions are combined, the
symmetry-dependent contributions cancel identically in every cumulant,
leaving precisely the same limiting cumulants as in the GinUE and GinSE. 

Such universality across different symmetry classes is not a property of
either species individually, but rather emerges only through their joint
fluctuations. At present, we do not have a conceptual explanation for
this remarkable cancellation mechanism.
\end{rem}

Another notable feature of Theorem~\ref{thm: cumulants} is that every
joint cumulant of order at least two is of order $\sqrt{N}$, regardless
of its order. More precisely, for every fixed pair $(j,k)$ with
$j+k\ge2$, we have 
\begin{equation} 
\kappa_{j,k}( \mathcal{N}_N^{\rm r}(R), \mathcal{N}_N^{\rm c}(R) )
=
O(\sqrt N).
\end{equation} 
This phenomenon has also been observed for the GinUE and GinSE, and may
be viewed as a higher-order analogue of hyperuniformity. Indeed,
hyperuniformity concerns only the growth of the second-order cumulant,
namely the number variance, whereas Theorem~\ref{thm: cumulants}
controls the entire hierarchy of higher-order cumulants. 

As an immediate consequence, since the higher-order cumulants become
negligible after the natural $N^{1/4}$ normalisation, we obtain the
following bivariate central limit theorem.

\begin{cor}[\textbf{Bivariate central limit theorem}] \label{cor: CLT}
 Under the assumptions of Theorem~\ref{thm: cumulants}, as $N\to\infty$,
\begin{equation}
\frac{1}{(2N/\pi)^{1/4}}
\left(
\begin{pmatrix}
\cN_N^{\rm r}(R)
\smallskip 
\\
\cN_N^{\rm c}(R)
\end{pmatrix}
-
\E
\begin{pmatrix}
\cN_N^{\rm r}(R)
\smallskip 
\\ 
\cN_N^{\rm c}(R)
\end{pmatrix}
\right) \to \textup{N}(\mathbf 0,\mathbf\Sigma_R),
\end{equation}
where the covariance matrix $\mathbf\Sigma_R$ is given according to the
corresponding scaling regime as follows.
    \begin{enumerate}[label=\textup{(\roman*)}]
        \item \textup{\textbf{(Bulk)}} If $R\in(0,1)$ is fixed, then
        \begin{equation}
            \boldsymbol{\Sigma}_R = \begin{pmatrix}
                (2-\sqrt{2})R & -(2-\sqrt{2})R
                \smallskip 
                \\
                -(2-\sqrt{2})R & 2R
            \end{pmatrix}.
        \end{equation}

        \item \textup{\textbf{(Edge)}} If $R=1-s/\sqrt{2N}$ for some $s\in\R$, then
        \begin{equation}
            \boldsymbol{\Sigma}_R = \begin{pmatrix}
                2-\sqrt{2} & -(2-\sqrt{2})
                \smallskip 
                \\
                -(2-\sqrt{2}) & 2-\sqrt{2}+ \sqrt{2}\,\mathcal{F}(s)
            \end{pmatrix},
        \end{equation}
        where $\mathcal{F}$ is given by \eqref{def of f number variance}.  
        \smallskip 
        \item \textup{\textbf{(Outside)}} If $R\in(1,\infty]$ is fixed, then
        \begin{equation}
            \boldsymbol{\Sigma}_R = \begin{pmatrix}
                2-\sqrt{2} & -(2-\sqrt{2})
                \smallskip 
                \\
                -(2-\sqrt{2}) & 2-\sqrt{2}
            \end{pmatrix}.
        \end{equation}
    \end{enumerate}
\end{cor}

As mentioned earlier, several works
\cite{FS23a,Si17,Fo24,GLX24,BMS25} have established Gaussian central limit
theorems for either the real eigenvalue counting function or the
non-real eigenvalue counting function of the GinOE. By contrast,
Corollary~\ref{cor: CLT} is, to the best of our knowledge, the first
result establishing a bivariate central limit theorem for their joint
fluctuations.

\begin{rem}[General domains beyond discs]
We conclude this section by commenting on the extension of our results to
more general counting domains. As mentioned above, for the GinUE, and also
for the purely non-real eigenvalue statistics of the GinOE, the number
variance can be analysed for fairly general domains. By contrast, the
precise asymptotic behaviour of higher-order cumulants for general counting
domains remains open even for the GinUE. Nevertheless, the corresponding central limit theorem for the GinUE follows from the number variance asymptotics and the general theory of determinantal point processes \cite{Sos02}. 

In the present work, our derivation of the precise asymptotics of the joint
cumulants relies heavily on the fact that the counting domain is a centred
disc. This is closely related to the determinantal structure that will be
introduced in the following section. In the disc setting, the relevant
matrix entries admit explicit representations in terms of incomplete gamma
functions, which allows us to exploit their uniform asymptotic expansions
and thereby obtain precise estimates throughout the different scaling
regimes. Although a finite-$N$ representation analogous to
Proposition~\ref{Prop_finite N formula} can still be formulated for a
general counting domain, the resulting matrix entries no longer possess
such an explicit structure, making a precise asymptotic analysis
considerably more challenging.

For the number variance, however, an alternative route may be available.
Following an approach similar to that used in \cite{MMO26}, one could work
directly with the one- and two-point correlation functions rather than with
the generating function. In the GinOE setting, this approach presents
additional difficulties compared with the GinUE. In particular, one has
to keep track separately of the real--real, complex--complex, and
real--complex contributions, the last of which encodes the covariance
between the real and non-real eigenvalue counts. Moreover, the transition
between the two-dimensional non-real spectrum and the one-dimensional real
spectrum requires a careful asymptotic analysis in a neighbourhood of the
real axis. Establishing the number-variance asymptotics for general
counting domains, especially those intersecting the real axis, therefore
appears to be an interesting direction for future study.
\end{rem}

\subsection*{Strategy of the proof.}
The starting point for the proof of Theorem~\ref{thm: cumulants} 
is the finite-$N$ determinantal representation of the joint
generating function established in
Proposition~\ref{Prop_finite N formula}.
This reduces the problem to the asymptotic analysis of trace products
involving the structured matrices
$\bfL_N(R)$, $\bfM_N(R)$, $\bfE_{N,1}(R)$, and $\bfE_{N,2}(R)$ defined in \eqref{eq: def of bfL and bfM} and \eqref{eq: def of bfE1 and bfE2}.
The first two matrices determine the leading contributions, while
the latter two are negligible at the relevant scale.
The resulting trace asymptotics, combined with the trace-log
expansion, yield the joint cumulant asymptotics.
The overall strategy is summarised in Figure~\ref{Fig_proof strategy}.

\subsection*{Organisation of the paper.}
In  Section~\ref{sec: exact formulas}, we derive the finite-$N$ representation that provides
the algebraic foundation for the subsequent analysis.
In Section~\ref{sec: asymptotic analysis}, we establish the asymptotic estimates for the
structured trace products.
Finally, in Section~\ref{sec: proof main rslt}, we combine these estimates to complete the proof of Theorem~\ref{thm: cumulants}.

\subsection*{Acknowledgements} The authors were supported by the National Research Foundation of Korea grant (RS-2025-00516909, RS-2026-25518141). We thank Gernot Akemann, Giorgio Cipolloni, Markus Ebke, Peter J. Forrester, Patrick Lopatto, and Gregory Schehr for their interest and helpful discussions.

\subsection*{AI use disclosure} We used ChatGPT 5.6 Pro for language polishing at the final stage of manuscript preparation and for checking the proofs. All mathematical arguments and proofs were independently developed and written by the authors.

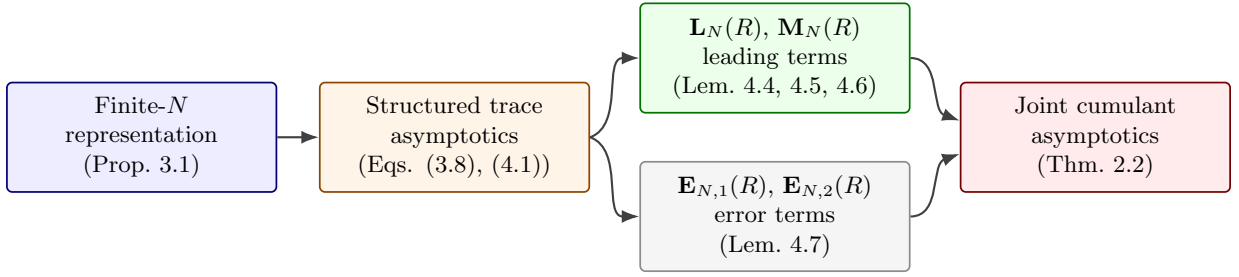
\begin{figure}[t]
\centering
\begin{tikzpicture}[
  box/.style={
    draw,
    rounded corners=2pt,
    minimum height=1.45cm,
    text width=3.35cm,
    align=center,
    inner sep=3pt,
    line width=0.65pt
  },
  boxFinite/.style={
    box,
    fill=blue!7,
    draw=blue!45!black
  },
  boxTrace/.style={
    box,
    fill=orange!9,
    draw=orange!55!black
  },
  boxMain/.style={
    box,
    fill=green!8,
    draw=green!45!black
  },
  boxError/.style={
    box,
    fill=gray!8,
    draw=black!45
  },
  boxFinal/.style={
    box,
    fill=red!8,
    draw=red!45!black
  },
  arr/.style={
    -Latex,
    line width=0.8pt,
    draw=black!75
  },
  every node/.style={
    font=\small
  }
]

\node[boxFinite] (finite)
{Finite-$N$ \\
representation\\
(Prop.~\ref{Prop_finite N formula})};

\node[boxTrace, right=0.55cm of finite] (traces)
{Structured trace\\
asymptotics \\ 
(Eqs. \eqref{eq for joint cumulants}, \eqref{eq: trace log expansion})};

\node[
  boxMain,
  right=0.65cm of traces,
  yshift=1.05cm
] (main)
{$\bfL_N(R),\,\bfM_N(R)$\\
leading terms
\\
(Lem.~\ref{lem: trace power bfL-n-R asymp}, ~\ref{lem: trace power bfM-n-R asymp}, ~\ref{lem: trace power mixed products only bfM bfL})};

\node[
  boxError,
  right=0.65cm of traces,
  yshift=-1.05cm
] (errors)
{$\bfE_{N,1}(R),\,\bfE_{N,2}(R)$\\
error terms
\\
(Lem.~\ref{lem: trace power mixed products including bfE})};

\path
    (main.east) -- (errors.east)
    coordinate[midway] (branchEast);

\node[
  boxFinal,
  right=0.65cm of branchEast
] (cumulants)
{Joint cumulant\\
asymptotics\\
(Thm.~\ref{thm: cumulants}) };

\draw[arr]
    (finite.east) -- (traces.west);

\draw[arr]
    (traces.east)
    to[out=20, in=180]
    (main.west);

\draw[arr]
    (traces.east)
    to[out=-20, in=180]
    (errors.west);

\draw[arr]
    (main.east)
    to[out=0, in=160]
    ([yshift=0.22cm]cumulants.west);

\draw[arr]
    (errors.east)
    to[out=0, in=200]
    ([yshift=-0.22cm]cumulants.west);

\end{tikzpicture}
\caption{
    Schematic diagram of the proof strategy for
    Theorem~\ref{thm: cumulants}.
}
\label{Fig_proof strategy}
\end{figure}

\section{Exact structured formulas for disc counting statistics} \label{sec: exact formulas}

The main purpose of this section is to derive an explicit formula (Proposition~\ref{Prop_finite N formula}) for
the joint cumulants that is amenable to asymptotic analysis. More
precisely, starting from the determinantal representation of the
generating function, we express the joint cumulants in terms of certain
structured matrices that will be analysed in the subsequent section. 

For this purpose, we recall that the lower incomplete gamma function
\cite[Chapter~8]{NIST} is defined by
\begin{equation}
    \gamma(a,z)
    :=
    \int_0^z t^{a-1}e^{-t}\,dt,
    \qquad a>0,\quad z\geq0.
\end{equation} 
We next introduce the following structured matrices, which will serve
as the main objects in our subsequent asymptotic analysis. 

\begin{defn}{Structured matrices}
Let $\bfL_N(R)$ and $\bfM_N(R)$ be the $n\times n$ matrices with
entries
\begin{equation}
\begin{split}
[\bfL_N(R)]_{j,k} &:= \frac{\gamma(2j-1,NR^2)}{\Gamma(2j-1)} \delta_{j,k},
\\
[\bfM_N(R)]_{j,k} &:= \frac{ \gamma(j+k-\tfrac{3}{2},N R^2)}{\sqrt{2\pi\Gamma(2j-1)\Gamma(2k-1)}}, \label{eq: def of bfL and bfM}
\end{split}
\end{equation}
where $\delta_{j,k}$ denotes the Kronecker delta.

We further define the $n\times n$ matrices $\bfE_{N,1}(R)$ and
$\bfE_{N,2}(R)$ by
 \begin{equation}
 \begin{split} 
        [\bfE_{N,1}(R)]_{j,k} &:= \frac{ 2^{j-\frac{1}{2}} N^{k-1} R^{2k-2} e^{-\frac{NR^2}{2}} \gamma(j-\tfrac{1}{2}, \tfrac{NR^2}{2})}{\sqrt{2\pi\Gamma(2j-1)\Gamma(2k-1)}},
        \\ 
        [\bfE_{N,2}(R)]_{j,k} &:= \frac{ E_N(j,k;R) }{\sqrt{2\pi\Gamma(2j-1)\Gamma(2k-1)}}.  \label{eq: def of bfE1 and bfE2} 
 \end{split}
    \end{equation}
Here, $\delta_{j,k}$ denotes the Kronecker delta, and
    \begin{equation}
        E_N(j,k;R):= -(NR^2)^{j+k-\frac{3}{2}} \int_0^\pi   \sin((2j-2k +1)\theta) e^{-NR^2 \cos 2\theta} \erfc(\sqrt{2N}R \sin\theta)\,d\theta .
    \end{equation} 
\end{defn}

We use the term ``structured matrices'' here in a broad sense:
the matrices above are not Toeplitz, Hankel, or of another classical
structured form, but their entries possess explicit algebraic and
special-function structures that will play a crucial role in the
subsequent analysis. Closely related matrix structures arise naturally
in the study of asymmetric random matrices; see e.g. 
\cite{AK07,KPTTZ15,BMS25,KA05}.

We now state the main finite-$N$ formula of this section, which provides an explicit representation of the joint cumulants of the numbers of real and non-real eigenvalues. 

\begin{prop}[\textbf{Finite-$N$ representation of joint cumulants}] \label{Prop_finite N formula}
Let $N=2n$ be a positive even integer and let $R\in[0,\infty]$. For $t,u \in \mathbb{R}$, define 
\begin{align}
  \bfA_N(t;R)
    &:=
    (e^{2t}-1)\bfM_N(R)
    -
    (e^{2t}-e^t)\bfE_{N,1}(R),  \label{def of AN t R in terms of M E}
   \\
      \bfB_N(u;R)
    &:=
    (e^{2u}-1)
    \Big(
        \bfL_N(R)
        -
        \bfM_N(R)
        -
        \bfE_{N,2}(R)
    \Big). \label{def of BN t R in terms of L E}
\end{align}  
Then we have
    \begin{align} \label{eq for generating function}
        \E\Big[e^{t \cN_N^{\rm r}(R) + u \cN_N^{\rm c}(R)}\Big] = \det \Big[ \mathrm{Id}_n + \bfA_N(t;R) + \bfB_N(u;R) \Big].
    \end{align}
Consequently, for nonnegative integers $p$ and $q$ with $p+q\geq1$, we have 
    \begin{equation}
        \begin{split} \label{eq for joint cumulants}
    \kappa_{p,q}
        \bigl(
            \cN_N^{\rm r}(R),
            \cN_N^{\rm c}(R)
        \bigr)
    &=    p!\,q!\,[t^p u^q] \Tr\log \Big[ \mathrm{Id}_{n} + \bfA_N(t;R) + \bfB_N(u;R) \Big]. 
        \end{split}
    \end{equation}
Here, $[t^pu^q]f(t,u)$ denotes the coefficient of $t^pu^q$
in the Taylor expansion of $f$ about $(t,u)=(0,0)$. 
\end{prop}

Note that \eqref{eq for joint cumulants} follows immediately from
\eqref{eq for generating function}, the definition
\eqref{def of joint cumulant in general} of the joint cumulants, and
the identity $\log\det A=\Tr\log A$. The rest of this section is devoted to the proof of \eqref{eq for generating function}. 

\begin{rem} As a particular case of \eqref{eq for joint cumulants}, taking the first derivatives at $t=u=0$ gives
    \begin{equation}
        \begin{split}
        \E\cN_N^{\rm r}(R)
        &=
        2\Tr\bfA_N^{(1)}(R)
        =
        2\Tr\bfM_N(R)-\Tr\bfE_{N,1}(R),
        \\
        \E\cN_N^{\rm c}(R)
        &=
        2\Tr\bfB_N^{(1)}(R)
      = 
        2\Tr\bfL_N(R)
        -
        2\Tr\bfM_N(R)
        -
        2\Tr\bfE_{N,2}(R).
        \end{split}
    \end{equation}
    Consequently, for the total counting function
    $\cN_N(R)=\cN_N^{\rm r}(R)+\cN_N^{\rm c}(R)$,
    \begin{equation}
        \E\cN_N(R)
        =
        2\Tr\bfL_N(R)
        -
        \Tr\bfE_{N,1}(R)
        -
        2\Tr\bfE_{N,2}(R).
    \end{equation}
    These identities recover the finite-$N$ formulas for the expected
    numbers of real and non-real eigenvalues obtained in
    \cite[Proposition~2.1 (iii)]{ABES23}.
\end{rem}

The algebraic formula in Proposition~\ref{Prop_finite N formula}
originates from the Pfaffian representation of the generating
function for the GinOE eigenvalues \cite{FN07,FN08,BS09,So07,SW08}. For this, we first introduce some notation. Define the functionals 
\begin{align}
    \funcA_{j,k}[f(x,y)] &:= \frac{1}{2} \int_\R dx \int_\R dy \, f(x,y) e^{-N\frac{x^2+y^2}{2}} q_{j-1}(x) q_{k-1}(y) \operatorname{sgn}(y-x), \label{def of mathcal A}
    \\
    \funcB_{j,k}[g(z)] &:= i \int_\C d^2z \, g(z) e^{-N\frac{z^2+\overline{z}^2}{2}} \erfc\big(\sqrt{2N} \, |\im z|\big) \, q_{j-1}(z) q_{k-1}(\overline{z}) \sgn(\im z), \label{def of mathcal B}
\end{align}
where
\begin{equation}
    q_{2j}(z) := z^{2j}, \qquad q_{2j+1}(z) = z^{2j+1} - \frac{2j}{N} z^{2j-1}.
\end{equation}
The polynomials $q_j$ form a family of skew-orthogonal polynomials
with respect to the skew-inner product associated with the GinOE;
see \cite{FN07,FN08} and \cite[Chapter 7]{BF25}.  

For the even-dimensional GinOE, the Pfaffian representation of the
moment generating function reduces to a determinant owing to the
checkerboard structure; see \cite[Proposition~2.1]{Si17}.
Analogous determinantal representations arise for a broader class of
real asymmetric random matrices; see e.g. \cite{BJLS25,BMS25} and the
references therein.
The resulting formula provides the starting point for the proof of
Proposition~\ref{Prop_finite N formula}.

\begin{lem}[cf. {\cite[Proposition 2.1]{Si17}}] \label{lem: generating func cN counting statistics}
 Let $N=2n$ be a positive even integer and let $R\in[0,\infty]$.
Then the generating function admits the representation
\eqref{eq for generating function}, where $\bfA_N(t;R)$ and
$\bfB_N(u;R)$ are $n\times n$ matrices with entries
    \begin{align}
        [\bfA_N(t;R)]_{j,k} &=  \frac{N^{j+k-\frac{1}{2}} \funcA_{2j-1,2k}[e^{t ( a(x;R) +  a(y;R) ) }-1] }{\sqrt{2\pi \Gamma(2j-1) \Gamma(2k-1)}}, \label{def of AN t R SOP}
        \\
        [\bfB_N(u;R)]_{j,k} &=  \frac{N^{j+k-\frac{1}{2}}\funcB_{2j-1,2k}[e^{2u b(z;R)}-1]}{\sqrt{2\pi \Gamma(2j-1) \Gamma(2k-1)}}. \label{def of BN t R SOP}
    \end{align}
    Here, $\funcA_{j,k}$ and $\funcB_{j,k}$ are defined in \eqref{def of mathcal A} and \eqref{def of mathcal B}, respectively, and
    \begin{equation}
    a(x;R):=\mathbbm{1}_{[-R,R]}(x), \qquad b(z;R):=\mathbbm{1}_{\D_R}(z). 
    \end{equation} 
\end{lem}

In view of Lemma~\ref{lem: generating func cN counting statistics},
to complete the proof of Proposition~\ref{Prop_finite N formula},
it remains to show that the skew-orthogonal polynomial representations
\eqref{def of AN t R SOP} and \eqref{def of BN t R SOP} reduce to the
structured matrix forms \eqref{def of AN t R in terms of M E} and
\eqref{def of BN t R in terms of L E}, respectively. 
For this purpose, write
\begin{align} \label{def of AN BN t R expansion}
    \bfA_N(t;R) = \sum_{\nu\geq1} \frac{\bfA_N^{(\nu)}(R)}{\nu!} 2^\nu t^\nu,
    \qquad
    \bfB_N(u;R) = \sum_{\nu\geq1} \frac{\bfB_N^{(\nu)}(R)}{\nu!} 2^\nu u^\nu,
\end{align}
where $\bfA_N^{(\nu)}(R)$ and $\bfB_N^{(\nu)}(R)$ are $n\times n$
matrices with entries 
\begin{align} \label{def of AN nu R}
    [\bfA_N^{(\nu)}(R)]_{j,k} &:= \frac{N^{j+k-\frac{1}{2}} \funcA_{2j-1,2k}[(a(x;R) + a(y;R))^\nu/2^\nu]}{\sqrt{2\pi\Gamma(2j-1)\Gamma(2k-1)}},
    \\  \label{def of BN nu R}
    [\bfB_N^{(\nu)}(R)]_{j,k} &:= \frac{N^{j+k-\frac{1}{2}}\funcB_{2j-1,2k}[b(z;R)]}{\sqrt{2\pi\Gamma(2j-1)\Gamma(2k-1)}} .
\end{align}
Here, we have used the fact that $b(z;R)^\nu=b(z;R)$ for every $\nu\geq1$.

Proposition~\ref{Prop_finite N formula} now follows from the following
decomposition.

\begin{lem} \label{lem: decomposition A-nu, B-nu, C-nu}
   Under the assumptions of Lemma~\ref{lem: generating func cN counting statistics}, for every positive integer $\nu$, we have
    \begin{align}
        \bfA_N^{(\nu)}(R) &= \bfM_N(R) - (1-2^{-\nu}) \bfE_{N,1}(R),
        \\
        \bfB_N^{(\nu)}(R) &= \bfL_N(R) - \bfM_N(R) - \bfE_{N,2}(R). 
    \end{align} 
Consequently, the skew-orthogonal polynomial representations
\eqref{def of AN t R SOP} and \eqref{def of BN t R SOP}
reduce to the structured matrix forms \eqref{def of AN t R in terms of M E} and \eqref{def of BN t R in terms of L E}, respectively. 
\end{lem}
\begin{proof}
By \eqref{def of mathcal A} and \eqref{def of AN nu R}, we have 
    \begin{equation}
        [\bfA_N^{(\nu)}(R)]_{j,k} = \frac{A_1 + 2^{-\nu} (A_2 + A_3)}{\sqrt{2\pi\Gamma(2j-1)\Gamma(2k-1)}}
    \end{equation}
    where
    \begin{align}
        A_\ell := \frac{N^{j+k-\frac{1}{2}}}{2} \iint_{D_\ell} dx \, dy \, e^{-N\frac{x^2+y^2}{2}} q_{2j-2}(x) q_{2k-1}(y) \sgn(y-x),
        \qquad
        \ell=1,2,3,
    \end{align}
with the domains of integration given by 
    \begin{align}
        D_1 :=  [-R,R]^2,
        \qquad
        D_2 := [-R,R] \times (\R \setminus [-R,R]),
        \qquad
        D_3 := (\R \setminus [-R,R]) \times [-R,R]. 
    \end{align}

    Since $q_{2k-1}$ is odd, we readily obtain
    \begin{align*}
        A_3 = \frac{N^{j+k-\frac{1}{2}}}{2} \bigg( \int_{-\infty}^{-R} - \int_R^\infty  \bigg) dx \int_{-R}^R dy \, e^{-N\frac{x^2 + y^2}{2}} q_{2j-2}(x) q_{2k-1}(y) = 0.
    \end{align*} 
Note here that for any positive integer $k$, 
\begin{equation} \label{eq: derivative SOP odd-even relation}
    e^{-N\frac{x^2}{2}}q_{2k-1}(x) = -\frac{1}{N} \frac{d}{dx}\Big[e^{-N\frac{x^2}{2}} q_{2k-2}(x) \Big].
\end{equation}
Since $q_{2j-2}$ is even, integration by parts
together with \eqref{eq: derivative SOP odd-even relation} gives 
    \begin{align*}
        A_2 &= \frac{N^{j+k-\frac{1}{2}}}{2} \int_{-R}^R dx \bigg( \int_R^\infty - \int_{-\infty}^{-R} \bigg) dy \, e^{-N\frac{x^2 + y^2}{2}} q_{2j-2}(x) q_{2k-1}(y)
        \\
        &= N^{j+k-\frac{3}{2}}e^{-\frac{NR^2}{2}} q_{2k-2}(R) \int_{-R}^R dx \, e^{-N\frac{x^2}{2}} q_{2j-2}(x) = 2^{j-\frac{1}{2}} N^{k-1} R^{2k-2} e^{-\frac{NR^2}{2}} \gamma(j-\tfrac{1}{2}, \tfrac{NR^2}{2}).
    \end{align*} 
Similarly,
    \begin{align*}
        A_1 &= \frac{N^{j+k-\frac{1}{2}}}{2} \int_{-R}^R dx \, \bigg( \int_x^R - \int_{-R}^x\bigg) dy \, e^{-N\frac{x^2+y^2}{2}} q_{2j-2}(x) q_{2k-1}(y)
        \\
        &= N^{j+k-\frac{3}{2}}\int_{-R}^R dx \, e^{-N\frac{x^2}{2}} q_{2j-2}(x) \Big( e^{-N\frac{x^2}{2}} q_{2k-2}(x) - e^{-\frac{NR^2}{2}} q_{2k-2}(R) \Big)
        = \gamma(j+k-\tfrac{3}{2}, NR^2) - A_2.
    \end{align*}
Combining these identities yields
\begin{align}
    A_1+2^{-\nu}(A_2+A_3) =\gamma(j+k-\tfrac32,NR^2) -(1-2^{-\nu})A_2,
\end{align}
which proves the asserted representation of
$\bfA_N^{(\nu)}(R)$.

Next, we prove the asserted representation of
$\bfB_N^{(\nu)}(R)$. 
Set
\begin{equation}
    W_N(z)
    :=
    \erfc\bigl(\sqrt{2N}\,|\im z|\bigr)
    e^{-N\frac{z^2+\bar z^2}{2}}.
\end{equation}
Then, by \eqref{def of mathcal B} and \eqref{def of BN nu R},
\begin{align*}
    [\bfB_N^{(\nu)}(R)]_{j,k}
    &=
    \frac{iN^{j+k-\frac12}}
    {\sqrt{2\pi\Gamma(2j-1)\Gamma(2k-1)}}
    \int_{\D_R}
    W_N(z)\,
    q_{2j-2}(z)q_{2k-1}(\bar z)
    \sgn(\im z)\,d^2z
    \\
    &=
    \frac{iN^{j+k-\frac12}}
    {\sqrt{2\pi\Gamma(2j-1)\Gamma(2k-1)}}
    \int_{\D_R\cap\HH}
    W_N(z)
    \Big[
        q_{2j-2}(z)q_{2k-1}(\bar z)
        -
        q_{2j-2}(\bar z)q_{2k-1}(z)
    \Big]
    \,d^2z.
\end{align*}
Now, we use Green's theorem in the complex form
\begin{equation}
    \int_D \partial_{\bar z} f(z)\,d^2z
    =
    \frac{1}{2i}\int_{\partial D} f(z)\,dz,
    \qquad
    \int_D \partial_z f(z)\,d^2z
    =
    -\frac{1}{2i}\int_{\partial D} f(z)\,d\bar z,
\end{equation}
where $\partial D$ is oriented counterclockwise. 
Note that by \eqref{eq: derivative SOP odd-even relation}, we have
\begin{align*}
    e^{-N\frac{\bar z^2}{2}}q_{2k-1}(\bar z)
    =
    -\frac{1}{N}\partial_{\bar z}
    \Big[
        e^{-N\frac{\bar z^2}{2}}q_{2k-2}(\bar z)
    \Big], \qquad 
    e^{-N\frac{z^2}{2}}q_{2k-1}(z)
    =
    -\frac{1}{N}\partial_z
    \Big[
        e^{-N\frac{z^2}{2}}q_{2k-2}(z)
    \Big].
\end{align*}
Moreover, for $z\in\HH$,
\begin{align*}
    \partial_{\bar z}
    \erfc\bigl(\sqrt{2N}\,\im z\bigr)
    =
    -i\sqrt{\frac{2N}{\pi}}\,
    e^{-2N(\im z)^2},
    \qquad 
    \partial_z
    \erfc\bigl(\sqrt{2N}\,\im z\bigr)
    =
    i\sqrt{\frac{2N}{\pi}}\,
    e^{-2N(\im z)^2}.
\end{align*}
Applying integration by parts and Green's theorem therefore gives
\begin{equation}
    [\bfB_N^{(\nu)}(R)]_{j,k}
    =
    \frac{B_1+B_2}
    {\sqrt{2\pi\Gamma(2j-1)\Gamma(2k-1)}},
\end{equation}
where
\begin{align}
    B_1
    &:=
    N^{j+k-1}\sqrt{\frac{2}{\pi}}
    \int_{ \D_R\cap\HH }
    e^{-N|z|^2}
    \Big[
        q_{2j-2}(z)q_{2k-2}(\bar z)
        +
        q_{2j-2}(\bar z)q_{2k-2}(z)
    \Big]
    \,d^2z,
    \\
    B_2
    & :=
    -\frac{N^{j+k-\frac32}}{2}
    \bigg[
        \int_{\partial ( \D_R\cap\HH )}
        W_N(z)\,
        q_{2j-2}(z)q_{2k-2}(\bar z)\,dz
        +
        \int_{\partial ( \D_R\cap\HH )}
        W_N(z)\,
        q_{2j-2}(\bar z)q_{2k-2}(z)\,d\bar z
    \bigg].
\end{align}

For $B_1$, passing to polar coordinates $z=re^{i\theta}$ gives
\begin{align*}
    B_1
    &=
    N^{j+k-1}\sqrt{\frac{2}{\pi}}
    \int_0^R dr\int_0^\pi d\theta\,
    e^{-Nr^2}r^{2j+2k-3}
    \Big(
        e^{2i(j-k)\theta}
        +
        e^{-2i(j-k)\theta}
    \Big)
    =
    \sqrt{2\pi}\,
    \gamma(2j-1, N R^2)\,
    \delta_{j,k}.
\end{align*}
For $B_2$, we decompose the boundary $\partial(\D_R\cap\HH)$
into the interval $[-R,R]$ and the semicircular arc
\begin{equation}
    C:=\partial(\D_R\cap\HH)\setminus[-R,R].
\end{equation}
The contribution from $[-R,R]$ is
\begin{align*}
    -N^{j+k-\frac32}
    \int_{-R}^R
    e^{-Nx^2}x^{2j+2k-4}\,dx
    =
    -\gamma\left(j+k-\tfrac32,NR^2\right).
\end{align*}
On the other hand, note that
\begin{equation*}
    z^{2j-2}\bar z^{2k-2}\,dz
    +
    \bar z^{2j-2}z^{2k-2}\,d\bar z
    =
    -2R^{2j+2k-3}
    \sin\bigl((2j-2k+1)\theta\bigr)\,d\theta.
\end{equation*}
Then by parametrising $z=Re^{i\theta}$, $0\leq\theta\leq\pi$, we obtain
\begin{align*}
    B_2
    &=
    -\gamma(j+k-\tfrac32,NR^2)
    -E_N(j,k;R).
\end{align*} Substituting these identities into the expression above  completes the proof.
\end{proof}

\section{Asymptotic analysis of structured trace powers} \label{sec: asymptotic analysis}

In this section, we develop the asymptotic analysis required to study
the large-$N$ behaviour of the joint cumulants. Our starting point is
the finite-$N$ representation \eqref{eq for joint cumulants} derived
in the previous section. Recall the matrix expansion
\begin{equation} \label{eq: trace log expansion}
    \Tr\log\bigl[\mathrm{Id}_n+X\bigr]
    =
    \sum_{m=1}^{\infty}
    \frac{(-1)^{m-1}}{m}\Tr(X^m),
\end{equation}
valid whenever the series is well defined.

In view of \eqref{eq for joint cumulants} and
\eqref{eq: trace log expansion}, the analysis of the joint cumulants
reduces to understanding the asymptotic behaviour of trace products
involving $\bfA_N(t;R)$ and $\bfB_N(u;R)$. By
\eqref{def of AN t R in terms of M E} and
\eqref{def of BN t R in terms of L E}, these can in turn be expressed
in terms of the four structured matrices
$\bfL_N(R)$, $\bfM_N(R)$, $\bfE_{N,1}(R)$, and $\bfE_{N,2}(R)$.
Thus, we are led to analyse trace products of the form
\begin{equation} \label{def of mixed trace powers}
    \Tr\bigg[
        \prod_{\nu=1}^p
        \bfE_{N,1}(R)^{a_\nu}
        \bfE_{N,2}(R)^{b_\nu}
        \bfM_N(R)^{c_\nu}
        \bfL_N(R)^{d_\nu}
    \bigg], 
\end{equation}
where $p$ is fixed and
$a_\nu,b_\nu,c_\nu,d_\nu$ are nonnegative integers.

The four structured matrices in
Proposition~\ref{Prop_finite N formula} play two distinct roles.
As we shall show, $\bfL_N(R)$ and $\bfM_N(R)$ determine the leading
contributions to the joint cumulants, whereas
$\bfE_{N,1}(R)$ and $\bfE_{N,2}(R)$ contribute only lower-order terms.
More precisely, the asymptotic analysis consists of the following
three ingredients:
\begin{itemize}
    \item \textbf{Main contributions.}
    We derive precise asymptotic formulas for the pure trace powers
    \begin{equation}
        \Tr\bigl[\bfL_N(R)^m\bigr],
        \qquad
        \Tr\bigl[\bfM_N(R)^m\bigr],
    \end{equation}
    which provide the leading-order contributions to the joint
    cumulants; see Lemmas~\ref{lem: trace power bfL-n-R asymp} and
    \ref{lem: trace power bfM-n-R asymp}.

    \smallskip

    \item \textbf{Mixed trace powers.}
    We show that mixed trace products involving only $\bfL_N(R)$ and
    $\bfM_N(R)$ reduce, up to a lower-order error, to the corresponding
    pure trace powers of $\bfM_N(R)$; see
    Lemma~\ref{lem: trace power mixed products only bfM bfL}.

    \smallskip

    \item \textbf{Error terms.}
    We show that any trace product containing at least one factor of
    $\bfE_{N,1}(R)$ or $\bfE_{N,2}(R)$ is negligible at the relevant
    scale; see
    Lemma~\ref{lem: trace power mixed products including bfE}.
\end{itemize}
Consequently, the leading-order asymptotics of the joint cumulants are
entirely determined by $\bfL_N(R)$ and $\bfM_N(R)$.

The section is organised as follows. In
Subsection~\ref{sec: matrix elements}, we establish entrywise estimates
for the four structured matrices. In
Subsection~\ref{sec: pure trace powers}, we analyse the pure trace
powers of $\bfL_N(R)$ and $\bfM_N(R)$, leading to
Lemmas~\ref{lem: trace power bfL-n-R asymp} and
\ref{lem: trace power bfM-n-R asymp}. Finally, in
Subsection~\ref{sec: mixed trace powers}, we study mixed trace products
and establish
Lemmas~\ref{lem: trace power mixed products only bfM bfL} and
\ref{lem: trace power mixed products including bfE}.
Together, these results provide the asymptotic input required for the
proof of the main results in the following section.

\subsection{Matrix elements} \label{sec: matrix elements}

Throughout the subsequent analysis, we will repeatedly use the uniform
asymptotic expansion of the incomplete gamma function. For convenience,
we recall the relevant expansion from \cite[Section~8.12]{NIST}. For
$a,z>0$, we have 
\begin{equation} \label{eq: asymp incomplete gamma}
    \frac{\gamma(a,z)}{\Gamma(a)} = \frac{1}{2} \operatorname{erfc}\Big( - \eta \sqrt{a/2} \Big) + F(a, \eta), 
\end{equation}
where
\begin{equation}
    \lambda := \frac{z}{a},
    \qquad
    \eta := \Big( 2 (\lambda - 1 - \log \lambda) \Big)^{1/2}. 
\end{equation}
Here, the branch of the square root is chosen so that
\begin{equation}
    \eta=\lambda-1+o(\lambda-1)
    \qquad
    \text{as }\lambda\to1.
\end{equation}
The remainder term $F(a,\eta)$ admits the asymptotic expansion
\begin{equation}
    F(a,\eta) \sim \frac{e^{-a \eta^2/2}}{\sqrt{2\pi a}} \sum_{k=0}^\infty c_k(\eta) a^{-k},
\end{equation}
where the coefficients $c_k(\eta)$ are explicitly computable.

Using this expansion, we first derive an entrywise asymptotic estimate
for $\bfM_N(R)$ defined in \eqref{eq: def of bfL and bfM}.

\begin{lem} \label{lem: element estimates bfM-n-R}
Suppose that $R\equiv R_N>0$ satisfies
\begin{equation}
    R_N\to r\in(0,\infty]
    \qquad
    \text{as }N\to\infty.
\end{equation}
For $j,k\in\{1,\ldots,n\}$, set
\begin{equation} \label{def of x and y}
    x:=\frac{2j-1}{NR^2},
    \qquad
    y:=\frac{2k-1}{NR^2}.
\end{equation}
Then, for every fixed $\epsilon\in(0,1)$, as $N\to\infty$,
\begin{equation}
    \begin{split}
    [\bfM_N(R)]_{j,k}
    =
    \sqrt{\frac{2}{\pi}}\,
    \frac{(xy)^{1/4}}{\sqrt{N}R(x+y)} 
    \exp\bigg[
        -\frac{NR^2}{2}
        \Big(
            x\log x
            +
            y\log y
            -
            (x+y)\log\frac{x+y}{2}
        \Big)
        +
        O\Big(
            \frac{1}{j}
            +
            \frac{1}{k}
        \Big)
    \bigg].
    \end{split}
\end{equation}
The error term is uniform over all $j,k\in\{1,\ldots,n\}$
satisfying $ 0<x,y\leq1-\epsilon.$ 
\end{lem}

\begin{proof}
Since $x,y\leq1-\epsilon$, we have
\begin{equation}
    j+k-\frac32
    =
    \frac{NR^2}{2}(x+y)-\frac12
    \leq
    (1-\epsilon)NR^2-\frac12.
\end{equation}
Thus, the argument $NR^2$ of the incomplete gamma function lies in the definition \eqref{eq: def of bfL and bfM} of $\bfM_N(R)$ uniformly to the right of its transition point. It follows from
\eqref{eq: asymp incomplete gamma} that, for some $c_\epsilon>0$,
\begin{equation}
    \gamma(j+k-\tfrac32,NR^2)
    =
    \Gamma(j+k-\tfrac32)
    \Big(1+O(e^{-c_\epsilon NR^2})\Big),
\end{equation}
uniformly in the stated range. Hence,
\begin{equation}
    [\bfM_N(R)]_{j,k}
    =
    \frac{
        \Gamma(j+k-\frac32)
    }{
        \sqrt{2\pi\Gamma(2j-1)\Gamma(2k-1)}
    }
    \Big(1+O(e^{-c_\epsilon NR^2})\Big).
\end{equation}
Recall that Stirling's approximation
\cite[Eqs.~(5.6.1), (5.11.1)]{NIST} gives
\begin{equation} \label{eq for Stirling}
    \Gamma(x)
    =
    \exp\Big[
        (x-\tfrac12)\log x
        -
        x
        +
        \tfrac12\log(2\pi)
        +
        O(1/x)
    \Big],
\end{equation}
where the error term is uniform for $x\geq\epsilon$, for any fixed
$\epsilon>0$. Applying \eqref{eq for Stirling} to the above ratio of
gamma functions and using
\begin{equation}
    2j-1=NR^2x,
    \qquad
    2k-1=NR^2y,
    \qquad
    j+k-1=\frac{NR^2}{2}(x+y),
\end{equation}
we obtain the desired asymptotic formula.
\end{proof}

We next establish an entrywise estimate for the matrix
$\bfE_{N,1}(R)$.

\begin{lem} \label{lem: element estimates bfE-n-R-1}
Under the assumptions of
Lemma~\ref{lem: element estimates bfM-n-R}, as $N\to\infty$, we have
\begin{equation}
    \begin{split}
    \big|[\bfE_{N,1}(R)]_{j,k}\big|
     \leq
    \frac{1}{\sqrt{\pi N}R}
    \Big(\frac{k}{j}\Big)^{1/4}
    \exp\bigg[
        -\frac{NR^2}{2}
        \bigl(y\log y-y+1\bigr)
        +
        O\Big(
            \frac{1}{j}
            +
            \frac{1}{k}
        \Big)
    \bigg],
    \end{split}
\end{equation}
where $y$ is given by \eqref{def of x and y}.  
The error term is uniform for all
$j,k\in\{1,\ldots,n\}$.
\end{lem}
\begin{proof}
Using $\gamma(a,z)\leq\Gamma(a)$ for $a,z>0$, it follows from
\eqref{eq: def of bfE1 and bfE2} that
\begin{equation*}
    \begin{split}
    \big|[\bfE_{N,1}(R)]_{j,k}\big|
    &\leq
    \frac{
        2^{j-\frac12}(NR^2)^{k-1}
        e^{-\frac{NR^2}{2}}
        \Gamma(j-\tfrac12)
    }{
        \sqrt{
            2\pi
            \Gamma(2j-1)
            \Gamma(2k-1)
        }
    }.
    \end{split}
\end{equation*}
Applying Stirling's approximation \eqref{eq for Stirling}, we obtain
\begin{equation*}
    \begin{split}
    \big|[\bfE_{N,1}(R)]_{j,k}\big|
    &\leq
    \frac{1}{
        \sqrt{\pi}
        (2j-1)^{1/4}
        (2k-1)^{1/4}
    } 
    \exp\bigg[
        -\frac{NR^2}{2}
        -(k-1)\log\Big(\frac{2k-1}{NR^2}\Big)
        +k-\frac12
        +
        O\Big(
            \frac1j+\frac1k
        \Big)
    \bigg].
    \end{split}
\end{equation*}
Using $y=(2k-1)/(NR^2)$ and rearranging the prefactor and
exponential terms, we obtain the assertion. 
\end{proof}

We end this subsection by establishing an entrywise estimate for the matrix $\bfE_{N,2}(R)$.

\begin{lem} \label{lem: element estimates bfE-n-R-2}
    Under the assumptions of Lemma~\ref{lem: element estimates bfM-n-R}, we have as $N\to\infty$ that
    \begin{equation}
    \begin{split}
        \big|[\bfE_{N,2}(R)]_{j,k} \big| &\leq \sqrt{\frac{2}{\pi}} \frac{ (2j-1)^{1/4}(2k-1)^{1/4} }{N R^2} 
        \exp\Big[ - \frac{N R^2}{2} (x \log x - x + y \log y - y +2) + O\Big(\frac{1}{j}+\frac{1}{k}\Big) \Big],
    \end{split}
    \end{equation}
    where $x$ and $y$ are given by \eqref{def of x and y}. 
    Here, the error term is uniform for all $j,k \in \{1, \ldots, n\}$.
\end{lem}

\begin{proof}
We first estimate $E_N(j,k;R)$ for arbitrary $R>0$. Set
\begin{equation*}
    \ell:=2j-2k+1,
    \qquad
    m:=|\ell|.
\end{equation*} 
Since $\ell$ is an odd integer, symmetry gives
\begin{equation*}
\begin{split}
\frac{E_N(j,k;R)}{(NR^2)^{j+k-\frac32}}
&=
-2\operatorname{sgn}(\ell)e^{-NR^2}
\int_0^{\pi/2}
\sin(m\theta)
e^{2NR^2\sin^2\theta}
\erfc(\sqrt{2N}R\sin\theta)\,d\theta
\\
&=
-\frac{4\operatorname{sgn}(\ell)e^{-NR^2}}{\sqrt{\pi}}
\int_0^\infty
e^{-t^2}
F(2\sqrt{2N}Rt,m)\,dt,
\end{split}
\end{equation*}
where
\begin{equation*}
    F(x,m)
    :=
    \int_0^{\pi/2}
    \sin(m\theta)e^{-x\sin\theta}\,d\theta.
\end{equation*}
Here, the second relation follows from the definition of the
complementary error function and Fubini's theorem. Consequently,
\begin{equation*}
\frac{|E_N(j,k;R)|}{(NR^2)^{j+k-\frac32}}
\leq
\frac{4e^{-NR^2}}{\sqrt{\pi}}
\int_0^\infty e^{-t^2}
|F(2\sqrt{2N}Rt,m)|\,dt.
\end{equation*}

We next estimate $F(x,m)$ for $x>0$. Integration by parts gives
\begin{equation*}
    \begin{split}
    |F(x,m)|
    &=
    \bigg|
        \frac{1-\cos(m\pi/2)}{m}e^{-x}
        +
        \frac{x}{m}
        \int_0^{\pi/2}
        \cos\theta\,e^{-x\sin\theta}
        \bigl(1-\cos(m\theta)\bigr)\,d\theta
    \bigg|
    \leq
    \frac{4}{m}.
    \end{split}
\end{equation*}
On the other hand, using
$|\sin(m\theta)|\leq m\sin\theta$ and
$\sin\theta\geq2\theta/\pi$ for $\theta\in[0,\pi/2]$, we obtain
\begin{equation*}
    |F(x,m)|
    \leq
    m\int_0^{\pi/2}
    \theta e^{-2x\theta/\pi}\,d\theta
    \leq
    \frac{\pi^2m}{4x^2}.
\end{equation*}
Combining these two estimates yields, for any $a>0$,
\begin{equation*}
    \begin{split}
    \int_0^\infty
    e^{-t^2}|F(at,m)|\,dt
    &\leq
    \int_0^\infty
    \min\bigg\{
        \frac4m,
        \frac{\pi^2m}{4a^2t^2}
    \bigg\}\,dt
    =
    \frac{2\pi}{a}.
    \end{split}
\end{equation*}
Taking $a=2\sqrt{2N}R$, we therefore obtain
\begin{equation*}
    \frac{|E_N(j,k;R)|}{(NR^2)^{j+k-\frac32}}
    \leq
    \frac{2\sqrt{2\pi}}{\sqrt{NR^2}}
    e^{-NR^2}.
\end{equation*}
Consequently,
\begin{equation}
    \big|[\bfE_{N,2}(R)]_{j,k}\big|
    \leq
    \frac{
        2(NR^2)^{j+k-2}e^{-NR^2}
    }{
        \sqrt{\Gamma(2j-1)\Gamma(2k-1)}
    }.
\end{equation}
Applying Stirling's approximation \eqref{eq for Stirling} and using \eqref{def of x and y}, we obtain the desired asymptotic behaviour. 
\end{proof}

\subsection{Pure trace powers} \label{sec: pure trace powers}

In this subsection, we analyse asymptotics for the trace powers.

We begin by recalling the first-order Euler--Maclaurin summation
formula, which will be used repeatedly throughout this subsection.
If $f\in C^1[a,b]$ for real numbers $a<b$, then
\begin{equation} \label{eq: Euler Maclaurin summation formula}
    \begin{split}
    \sum_{a<j\leq b} f(j)
    &=
    \int_a^b f(x)\,dx
    +
    \int_a^b
    (x-\lfloor x\rfloor)f'(x)\,dx
  +
    (a-\lfloor a\rfloor)f(a)
    -
    (b-\lfloor b\rfloor)f(b).
    \end{split}
\end{equation}
We also use the shorthand notation
\begin{equation}
    x\wedge y:=\min\{x,y\},
    \qquad
    x\vee y:=\max\{x,y\}.
\end{equation}

We begin with the pure trace powers of the diagonal matrix
$\bfL_N(R)$ defined in \eqref{eq: def of bfL and bfM}.

\begin{lem} \label{lem: trace power bfL-n-R asymp}
Let $N=2n$ be a positive even integer. Suppose that
$r\in(0,\infty)$ and
\begin{equation}
    R
    =
    r-\frac{s}{\sqrt{2N}}
    +
    o\bigl(N^{-1/2}\bigr)
 \label{def of R scaling in lemma}
 \end{equation}
as $N\to\infty$, for some fixed $s\in\R$. Then, for every fixed
positive integer $m$,
\begin{equation} \label{eq: trace power bfL asymp}
    \Tr\big[\bfL_N(R)^m\big]
    =
    (r^2\wedge1)n
    +
    \alpha_{m,r,s}\sqrt n
    +
    o(\sqrt n),
\end{equation}
where
\begin{equation}
    \alpha_{m,r,s}
    :=
    \begin{cases}
    \displaystyle
    r\int_{-\infty}^{\infty}
    \bigg[
        \bigg(\frac{\erfc(x)}{2}\bigg)^m
        -
        \mathbbm{1}_{\{x<s\}}
    \bigg]\,dx,
    & 0<r<1,
    \\[3mm]
    \displaystyle
    \int_{-\infty}^{s}
    \bigg[
        \bigg(\frac{\erfc(x)}{2}\bigg)^m
        -
        1
    \bigg]\,dx,
    & r=1,
    \\[3mm]
    0,
    & 1<r<\infty.
    \end{cases}
\end{equation}
If $R\to\infty$, the same conclusion holds with
$(r^2\wedge1)n$ interpreted as $n$ and
$\alpha_{m,\infty,s}:=0$.
\end{lem}

\begin{proof}
  The case $R\to\infty$ follows immediately from
\eqref{eq: asymp incomplete gamma}, so we assume that
$r\in(0,\infty)$.
 
    Since $\bfL_N(R)$ is diagonal, we have
    \begin{align*}
    \begin{split}
        \Tr[\bfL_N(R)^m] &= \sum_{j=1}^n \Big( \frac{\gamma(2j-1, NR^2)}{ \Gamma(2j-1)} \Big)^m
        = \lfloor n R^2 \rfloor \wedge n + \sum_{j=1}^n \bigg[ \Big( \frac{\gamma(2j-1,N R^2)}{ \Gamma(2j-1)} \Big)^m - \mathbbm{1}_{\{j \leq n R^2\}} \bigg].
    \end{split}
    \end{align*}
    The summands in the second term are concentrated near $j=nR^2$. To isolate this transition region, define
    \begin{equation} \label{eq: def tau-n-pm}
        \tau_n^\pm:= \Big\lfloor n R^2 \pm \sqrt{n} (\log n)^2 \Big\rfloor  \wedge n.
    \end{equation} 
    The uniform asymptotic expansion \eqref{eq: asymp incomplete gamma} implies that, for some $c>0$,
\begin{equation*}
    \begin{split}
\bigg(
        \sum_{j=1}^{\tau_n^-}
        +
        \sum_{j=\tau_n^++1}^{n}
    \bigg)
    \bigg[
        \bigg(
            \frac{\gamma(2j-1,NR^2)}{\Gamma(2j-1)}
        \bigg)^m
        -
        \mathbbm{1}_{\{j\leq nR^2\}}
    \bigg]
=
    O\bigl(
        n e^{-c(\log n)^4}
    \bigr).
    \end{split}
\end{equation*}

We next apply the first-order Euler--Maclaurin summation formula to
the transition region. Since $\tau_n^-$ and $\tau_n^+$ are integers,
the endpoint correction terms vanish, and hence 
    \begin{align*}
        \sum_{j=\tau_n^-+1}^{\tau_n^+}  \Big( \frac{\gamma(2j-1,N R^2)}{ \Gamma(2j-1)} \Big)^m &= \int_{\tau_n^-}^{\tau_n^+} \Big( \frac{\gamma(2x-1,N R^2)}{ \Gamma(2x-1)} \Big)^m dx + \int_{\tau_n^-}^{\tau_n^+} (x - \lfloor x \rfloor) \frac{d}{dx}\bigg[ \Big( \frac{\gamma(2x-1,N R^2)}{ \Gamma(2x-1)} \Big)^m \bigg] dx. 
    \end{align*}

For the first integral, set
\begin{equation}
    y
    :=
    \frac{x-(NR^2+1)/2}{\sqrt{nR^2}},
    \qquad
    \widetilde{\tau}_n^\pm
    :=
    \frac{
        \tau_n^\pm-(NR^2+1)/2
    }{
        \sqrt{nR^2}
    }.
\end{equation}
It follows from \eqref{eq: asymp incomplete gamma} that 
    \begin{align}
        \int_{\tau_n^-}^{\tau_n^+} \Big( \frac{\gamma(2x-1,NR^2)}{ \Gamma(2x-1)} \Big)^m dx = \sqrt{n R^2}\int_{\widetilde{\tau}_n^-}^{\widetilde{\tau}_n^+} \Big( \frac{\erfc(y)}{2} \Big)^m dy + O(1). 
    \end{align} 

For the second integral, we have 
    \begin{align*}
        &\quad \Bigg|\int_{\tau_n^-}^{\tau_n^+} (x - \lfloor x \rfloor) \frac{d}{dx}\bigg[ \Big( \frac{\gamma(2x-1,NR^2)}{ \Gamma(2x-1)} \Big)^m \bigg] dx \Bigg|
        \\
        &\leq \int_{\tau_n^-}^{\tau_n^+} \Bigg| \frac{d}{dx}\bigg[ \Big( \frac{\gamma(2x-1,NR^2)}{ \Gamma(2x-1)} \Big)^m \bigg] \Bigg| dx 
        =   \Big( \frac{\gamma(2\tau_n^+-1,NR^2)}{ \Gamma(2\tau_n^+-1)} \Big)^m  -  \Big( \frac{\gamma(2\tau_n^--1,NR^2)}{ \Gamma(2\tau_n^--1)} \Big)^m  \leq 1.
    \end{align*}
Here, we have used the fact that the regularised incomplete gamma function $\gamma(x,NR^2)/\Gamma(x)$ is decreasing and takes values in $(0,1)$. 
Combining these estimates and absorbing the discrepancy between
$j\leq nR^2$ and $y<0$ into the $O(1)$ term, we obtain 
    \begin{equation}
        \Tr[\bfL_N(R)^m] = \lfloor n R^2 \rfloor \wedge n  + \sqrt{nR^2}\int_{\widetilde{\tau}_n^-}^{\widetilde{\tau}_n^+} \bigg[ \Big( \frac{\erfc(y)}{2} \Big)^m - \mathbbm{1}_{\{y<0\}} \bigg] dy + O(1).
    \end{equation}
    It remains to evaluate the limiting integration range. Under the
scaling \eqref{def of R scaling in lemma} we have
\begin{equation}
    \widetilde{\tau}_n^-\to-\infty,
    \qquad
    \widetilde{\tau}_n^+\to
    \begin{cases}
        +\infty, & r\in(0,1), \smallskip 
        \\
        s, & r=1.
    \end{cases}
\end{equation}
For $r>1$, both $\tau_n^-$ and $\tau_n^+$ equal $n$ for all
sufficiently large $n$, so the transition contribution vanishes.
Moreover,
\begin{equation}
    nR^2=nr^2-rs\sqrt n+o(\sqrt n).
\end{equation}
Substituting these asymptotics into the preceding expression and
replacing the integration endpoints by their respective limits yields
\eqref{eq: trace power bfL asymp}.
\end{proof}

\begin{lem} \label{lem: trace power bfM-n-R asymp}
    Under the assumptions of Lemma~\ref{lem: trace power bfL-n-R asymp}, we have as $N = 2n\to\infty$ that
    \begin{equation} \label{eq: trace power bfM asymp}
        \Tr [\bfM_N(R)^m] = (r \wedge 1) \sqrt{\frac{n}{\pi m}} + o(\sqrt{n}). 
    \end{equation}
\end{lem}
\begin{proof} 
We first claim that
\begin{equation} \label{eq: trace power bfM-n-R diagonal concent}
    \Tr[\bfM_N(R)^m]
    =
    \frac{1}{(2\pi)^{m/2}}
    \sum_{j_1=1}^n
    \sum_{j_2,\ldots,j_m}
    \prod_{\ell=1}^m
    \frac{
        \gamma(j_\ell+j_{\ell+1}-\frac32,NR^2)
    }{
        \sqrt{
            \Gamma(2j_\ell-1)
            \Gamma(2j_{\ell+1}-1)
        }
    }
    +
    O(e^{-c(\log n)^2}),
\end{equation}
where the inner sum runs over all integers
$j_2,\ldots,j_m$ satisfying
\begin{equation}
    \vartheta_n^-(j_1)+1
    \leq
    j_\ell
    \leq
    \vartheta_n^+(j_1),
    \qquad
    \ell=2,\ldots,m,
\end{equation}
with
\begin{equation} \label{eq: def vartheta-n-pm}
    \vartheta_n^-(j)
    :=
    0\vee\lfloor j-\sqrt n\log n\rfloor,
    \qquad
    \vartheta_n^+(j)
    :=
    \lfloor j+\sqrt n\log n\rfloor\wedge n.
\end{equation}
Here, we adopt the cyclic convention $j_{m+1}\equiv j_1$.

From the definition \eqref{eq: def of bfL and bfM}, we have
    \begin{align} \label{eq: trace power bfM large summation}
        \Tr[\bfM_N(R)^m] = \frac{1}{(2\pi)^{m/2}} \sum_{j_1=1}^n \cdots \sum_{j_m=1}^n \prod_{\ell=1}^m \frac{\gamma(j_\ell+ j_{\ell+1}-\frac{3}{2}, N R^2)}{\sqrt{ \Gamma(2j_\ell-1) \Gamma(2j_{\ell+1}-1)}}. 
    \end{align}
To prove \eqref{eq: trace power bfM-n-R diagonal concent}, it suffices
to show that the contribution of those index tuples excluded from the
restricted summation is exponentially small.

We first establish a uniform bound for the gamma ratio. For
$x,y\geq1/2$ and $\rho>0$, we have
\begin{equation} \label{eq: trace power bfM summand bounds}
    0
    <
    \frac{
        \gamma(x+y-\tfrac12,\rho)
    }{
        \sqrt{\Gamma(2x)\Gamma(2y)}
    }
    \leq
    \frac{
        \Gamma(x+y-\tfrac12)
    }{
        \sqrt{\Gamma(2x)\Gamma(2y)}
    }
    \leq
    \sqrt{\pi}, 
\end{equation}
where the last inequality follows from the monotonicity of
$\Gamma(z-\tfrac12)/\Gamma(z)$ for $z\geq1$ and the log-convexity
of the gamma function.
Moreover, Stirling's approximation \eqref{eq for Stirling} gives
\begin{equation*}
    \frac{
        \Gamma(x+y-\tfrac12)
    }{
        \sqrt{\Gamma(2x)\Gamma(2y)}
    }
    =
    \exp\bigg[
        (x+y)\log(x+y)
        -
        x\log(2x)
        -
        y\log(2y)
        +
        O(\log(x+y))
    \bigg],
\end{equation*}
uniformly for $x,y\geq1/2$.

    Let us assume $1/2 \leq x,y\leq n$ while $|x-y| \geq \frac{1}{m}\sqrt{n} \log n$.
    Then, using convexity of $f(x) = x \log x$, it follows from the Taylor expansion that
    \begin{equation}
        (x+y)\log(x+y) - x \log(2x) - y \log(2y) \leq -2c (\log n)^2
    \end{equation}
    for some $m$-dependent constant $c>0$.
    Therefore,
    \begin{equation} \label{eq: bound on ratio of gamma function}
        \frac{\Gamma(x+y-\frac{1}{2})}{\sqrt{\Gamma(2x) \Gamma(2y)}} \leq \exp\Big[ - 2c (\log n)^2 + O(\log n) \Big]
    \end{equation}
    for some $m$-dependent constant $c>0$ with the error term being uniform.
    
We now return to the summation in
\eqref{eq: trace power bfM large summation}. Suppose that an index
tuple $(j_1,\ldots,j_m)\in\{1,\ldots,n\}^m$ contains some $j_k$
outside the interval
\begin{equation*}
    [\vartheta_n^-(j_1)+1,\vartheta_n^+(j_1)].
\end{equation*}
Then
\begin{equation*}
    |j_k-j_1|
    \geq
    \sqrt n\log n-O(1).
\end{equation*}
Hence, by the triangle inequality and the cyclic convention, for all
sufficiently large $n$, there exists
$\ell\in\{1,\ldots,m\}$ such that
\begin{equation*}
    |j_\ell-j_{\ell+1}|
    >
    \frac{1}{2m}\sqrt n\log n.
\end{equation*}
It follows from \eqref{eq: trace power bfM summand bounds} and
\eqref{eq: bound on ratio of gamma function} that
\begin{equation*}
    \prod_{\ell=1}^m
    \frac{
        \gamma(j_\ell+j_{\ell+1}-\frac32,NR^2)
    }{
        \sqrt{
            \Gamma(2j_\ell-1)
            \Gamma(2j_{\ell+1}-1)
        }
    }
    \leq
    \exp\Big[
        -2c(\log n)^2+O(\log n)
    \Big]
\end{equation*}
for some constant $c>0$ depending only on $m$. Since the number of
such index tuples is at most $n^m$, their total contribution is bounded
by
\begin{equation*}
    n^m
    \exp\Big[
        -2c(\log n)^2+O(\log n)
    \Big]
    =
    O\bigl(e^{-c'(\log n)^2}\bigr)
\end{equation*}
for some $c'>0$. This proves
\eqref{eq: trace power bfM-n-R diagonal concent}.

\medskip

   We now derive the desired asymptotic formula from
\eqref{eq: trace power bfM-n-R diagonal concent}.
For a positive integer $K$, define the $K\times K$ matrix
\begin{equation} \label{def of MK infty}
    \bfM_K^{(\infty)}
    :=
    \bigg[
        \frac{
            \Gamma(j+k-\frac32)
        }{
            \sqrt{
                2\pi
                \Gamma(2j-1)
                \Gamma(2k-1)
            }
        }
    \bigg]_{j,k=1}^K.
\end{equation}
We use the following result from
\cite[Lemma~2.3]{KPTTZ15}: for every fixed positive integer $m$,
\begin{equation} \label{KPTTZ15 formula}
    \Tr\big[
        \big(\bfM_K^{(\infty)}\big)^m
    \big]
    =
    \sqrt{\frac{K}{\pi m}}
    +
    o(\sqrt K)
\end{equation}
as $K\to\infty$.

Recall $\tau_n^\pm$ and $\vartheta_n^\pm$ from
\eqref{eq: def tau-n-pm} and
\eqref{eq: def vartheta-n-pm}, respectively, and define
\begin{equation} \label{def of Kn pm}
    K_n^-:=\tau_n^-,
    \qquad
    K_n^+
    :=
    \big(
        \tau_n^+
        +
        \lceil\sqrt n\log n\rceil
    \big)
    \wedge n.
\end{equation}
We first derive an upper bound. By
\eqref{eq: trace power bfM-n-R diagonal concent}, the indices in the
restricted summation satisfy
\begin{equation*}
    |j_\ell-j_1|
    \leq
    \sqrt n\log n+O(1),
    \qquad
    \ell=2,\ldots,m.
\end{equation*}
If $j_1>\tau_n^+$, then, unless $\tau_n^+=n$, the corresponding shape
parameters satisfy
\begin{equation*}
    j_\ell+j_{\ell+1}-\frac32
    \geq
    NR^2+c\sqrt n(\log n)^2
\end{equation*}
for all sufficiently large $n$ and some $c>0$. It follows from
\eqref{eq: asymp incomplete gamma} and
\eqref{eq: trace power bfM summand bounds} that the total contribution
of these terms is $O(e^{-c(\log n)^2})$.

For the remaining terms, $j_1\leq\tau_n^+$ implies
$j_\ell\leq K_n^+$ for every $\ell$. Since
$\gamma(a,z)\leq\Gamma(a)$, enlarging the summation range gives
\begin{equation*}
    \Tr\big[\bfM_N(R)^m\big]
    \leq
    \Tr\big[
        \big(\bfM_{K_n^+}^{(\infty)}\big)^m
    \big]
    +
    O(e^{-c(\log n)^2}).
\end{equation*}

For the lower bound, observe that, uniformly for
$j,k\leq K_n^-$,
\begin{equation*}
    j+k-\frac32
    \leq
    NR^2-c\sqrt n(\log n)^2
\end{equation*}
for some $c>0$. Hence,
\eqref{eq: asymp incomplete gamma} gives
\begin{equation*}
    \gamma\Big(j+k-\frac32,NR^2\Big)
    =
    \Gamma\Big(j+k-\frac32\Big)
    \Big(
        1+O(e^{-c(\log n)^2})
    \Big)
\end{equation*}
uniformly in this range. Restricting all indices in the trace
expansion to $\{1,\ldots,K_n^-\}$ therefore yields
\begin{equation*}
    \Tr\big[\bfM_N(R)^m\big]
    \geq
    \Big(
        1-O(e^{-c(\log n)^2})
    \Big)
    \Tr\big[
        \big(\bfM_{K_n^-}^{(\infty)}\big)^m
    \big].
\end{equation*}

Note that by \eqref{def of R scaling in lemma} and \eqref{def of Kn pm}, we have
\begin{equation*}
    \frac{K_n^\pm}{n}
    \longrightarrow
    r^2\wedge1,
    \qquad
    \sqrt{K_n^\pm}
    =
    (r\wedge1)\sqrt n
    +
    o(\sqrt n).
\end{equation*}
Applying \eqref{KPTTZ15 formula} to $K_n^-$ and $K_n^+$, we obtain
\begin{equation*}
    \Tr\big[
        \big(\bfM_{K_n^\pm}^{(\infty)}\big)^m
    \big]
    =
    (r\wedge1)
    \sqrt{\frac{n}{\pi m}}
    +
    o(\sqrt n).
\end{equation*}
The conclusion now follows from the preceding upper and lower bounds. 
\end{proof}

\subsection{Mixed trace powers} \label{sec: mixed trace powers}

We now turn to the mixed trace products appearing in
\eqref{def of mixed trace powers}. The aim of this subsection is to
show that, at the scale relevant to the joint cumulants, their
asymptotic behaviour is entirely determined by the pure trace powers
of $\bfM_N(R)$ and $\bfL_N(R)$ studied in the previous subsection.
More precisely, mixed products involving only $\bfM_N(R)$ and
$\bfL_N(R)$ reduce, to leading order, to the corresponding pure trace
powers of $\bfM_N(R)$, whereas any product containing at least one
factor of either $\bfE_{N,1}(R)$ or $\bfE_{N,2}(R)$ is negligible.
These two statements are established in Lemmas~\ref{lem: trace power mixed products only bfM bfL} and ~\ref{lem: trace power mixed products including bfE}. 

\begin{lem} \label{lem: trace power mixed products only bfM bfL}
  Under the assumptions of Lemma~\ref{lem: trace power bfL-n-R asymp}, let $p$ be a fixed positive integer, and let $a_\nu,b_\nu$ be fixed nonnegative integers for $\nu=1,\ldots,p$, such that $a_1 + \cdots +a_p \geq 1$.
    Then, we have as $N = 2n\to\infty$ that
    \begin{align}
        \Tr\bigg[\prod_{\nu=1}^p\bfM_N(R)^{a_\nu} \bfL_N(R)^{b_\nu} \bigg] = \Tr\big[ \bfM_N(R)^{a_1 + \cdots +a_p} \big] + o(\sqrt{n}).
    \end{align}
Here, $\prod_{\nu=1}^p\bfQ_\nu$ denotes the ordered matrix product
$\bfQ_1\bfQ_2\cdots\bfQ_p$.
\end{lem}
\begin{proof}
Since $\bfM_N(R)$ has nonnegative entries and $\bfL_N(R)$ is
diagonal with diagonal entries in $[0,1]$, expanding the trace as a
sum over indices gives the termwise bound
\begin{equation} \label{mixed LM upper}
    \Tr\bigg[
        \prod_{\nu=1}^p
        \bfM_N(R)^{a_\nu}\bfL_N(R)^{b_\nu}
    \bigg]
    \leq
    \Tr\big[
        \bfM_N(R)^{a_1+\cdots+a_p}
    \big]. 
\end{equation}
Thus, it remains to establish the corresponding lower bound up to an
$o(\sqrt n)$ error.

 For the lower bound, recall $K_n^-$ from \eqref{def of Kn pm}.
As in the lower-bound argument in the proof of
Lemma~\ref{lem: trace power bfM-n-R asymp}, uniformly for
$1\leq j,k\leq K_n^-$, we have  
\begin{equation*}
    [\bfM_N(R)]_{j,k}
    =
    [\bfM_{K_n^-}^{(\infty)}]_{j,k}
    \Big(
        1+O(e^{-c(\log n)^2})
    \Big),
\end{equation*}
where $\bfM_K^{(\infty)}$ is defined in \eqref{def of MK infty}.
Similarly,
\begin{equation*}
    [\bfL_N(R)]_{j,j}
    =
    1+O(e^{-c(\log n)^2}).
\end{equation*}
Since all matrix entries involved are nonnegative, restricting all
summation indices in the trace expansion to
$\{1,\ldots,K_n^-\}$ gives
\begin{equation*}
    \begin{split}
    \Tr\bigg[
        \prod_{\nu=1}^p
        \bfM_N(R)^{a_\nu}\bfL_N(R)^{b_\nu}
    \bigg] \geq
    \Big(
        1-O(e^{-c(\log n)^2})
    \Big)
    \Tr\bigg[
        \big(
            \bfM_{K_n^-}^{(\infty)}
        \big)^{a_1+\cdots+a_p}
    \bigg].
    \end{split}
\end{equation*}
By \eqref{KPTTZ15 formula} and \eqref{def of Kn pm}, we have
\begin{equation*}
    \Tr\bigg[
        \big(
            \bfM_{K_n^-}^{(\infty)}
        \big)^{a_1+\cdots+a_p}
    \bigg]
    =
    (r\wedge1)
    \sqrt{
        \frac{n}{
            \pi(a_1+\cdots+a_p)
        }
    }
    +
    o(\sqrt n).
\end{equation*}
On the other hand, Lemma~\ref{lem: trace power bfM-n-R asymp} gives
\begin{equation*}
    \Tr\big[
        \bfM_N(R)^{a_1+\cdots+a_p}
    \big]
    =
    (r\wedge1)
    \sqrt{
        \frac{n}{
            \pi(a_1+\cdots+a_p)
        }
    }
    +
    o(\sqrt n).
\end{equation*}
Consequently,
\begin{equation*}
    \Tr\bigg[
        \big(
            \bfM_{K_n^-}^{(\infty)}
        \big)^{a_1+\cdots+a_p}
    \bigg]
    =
    \Tr\big[
        \bfM_N(R)^{a_1+\cdots+a_p}
    \big]
    +
    o(\sqrt n).
\end{equation*}
Therefore, we obtain
\begin{equation} \label{mixed LM lower}
    \Tr\bigg[
        \prod_{\nu=1}^p
        \bfM_N(R)^{a_\nu}\bfL_N(R)^{b_\nu}
    \bigg]
    \geq
    \Tr\big[
        \bfM_N(R)^{a_1+\cdots+a_p}
    \big]
    +
    o(\sqrt n).
\end{equation}
Combining \eqref{mixed LM upper} and \eqref{mixed LM lower}, the lemma follows. 
\end{proof}

Next, we show that any mixed trace power involving
$\bfE_{N,1}(R)$ or $\bfE_{N,2}(R)$ is negligible at the
leading order.

\begin{lem} \label{lem: trace power mixed products including bfE}
Under the assumptions of
Lemma~\ref{lem: trace power bfL-n-R asymp}, let $p$ be a fixed
positive integer, and let $a_\nu,b_\nu,c_\nu,d_\nu$ be fixed
nonnegative integers for $\nu=1,\ldots,p$. Suppose that
\begin{equation}
    \sum_{\nu=1}^p(a_\nu+b_\nu)\geq1.
\end{equation}
Then, as $N=2n\to\infty$, we have
\begin{equation}
    \Tr\bigg[
        \prod_{\nu=1}^p
        \bfE_{N,1}(R)^{a_\nu}
        \bfE_{N,2}(R)^{b_\nu}
        \bfM_N(R)^{c_\nu}
        \bfL_N(R)^{d_\nu}
    \bigg]
    =
    o(\sqrt n).
\end{equation}
\end{lem}

\begin{proof}
The main idea of the proof is a localisation argument. The presence of a factor of either $\bfE_{N,1}(R)$ or $\bfE_{N,2}(R)$ forces the associated indices to lie near the transition point $nR^2$, while the off-diagonal decay of $\bfM_N(R)$ propagates this localisation through the remaining factors in the matrix product. Once all indices are restricted to a window of size $O(\sqrt n\log n)$ around $nR^2$, uniform $O(n^{-1/2})$ bounds on the relevant matrix entries yield the desired $o(\sqrt n)$ estimate.

Set
\begin{equation*}
    q
    :=
    \sum_{\nu=1}^p
    (a_\nu+b_\nu+c_\nu).
\end{equation*}
By assumption, $q\geq1$. After expanding the matrix powers and using
the fact that $\bfL_N(R)$ is diagonal with diagonal entries in
$[0,1]$, the absolute value of the trace is bounded by an index sum
associated with $q$ ordered factors
\begin{equation*}
    \bfQ_1,\ldots,\bfQ_q
    \in
    \{
        \bfE_{N,1}(R),
        \bfE_{N,2}(R),
        \bfM_N(R)
    \}.
\end{equation*}
More precisely, with the cyclic convention $j_{q+1}\equiv j_1$,
\begin{equation*}
    \bigg|
        \Tr\bigg[
            \prod_{\nu=1}^p
            \bfE_{N,1}(R)^{a_\nu}
            \bfE_{N,2}(R)^{b_\nu}
            \bfM_N(R)^{c_\nu}
            \bfL_N(R)^{d_\nu}
        \bigg]
    \bigg|
    \leq
    \sum_{j_1,\ldots,j_q=1}^n
    \prod_{\ell=1}^q
    \big|
        [\bfQ_\ell]_{j_\ell,j_{\ell+1}}
    \big|.
\end{equation*}
Since at least one of the factors is either $\bfE_{N,1}(R)$ or
$\bfE_{N,2}(R)$, cyclicity of the trace allows us to assume that
$\bfQ_q$ is such a factor.

We first consider the case $r>1$, including the case $R\to\infty$. 
Since $R\to r>1$, for all $1\leq j,k\leq n$, the variables $x$ and $y$ in \eqref{def of x and y} are uniformly bounded away from
$1$ for $1\leq j,k\leq n$. Hence,
Lemmas~\ref{lem: element estimates bfE-n-R-1} and
\ref{lem: element estimates bfE-n-R-2} imply that the entries of
$\bfE_{N,1}(R)$ and $\bfE_{N,2}(R)$ are uniformly exponentially
small. Together with the uniform bounds for the remaining factors,
this gives 
\begin{equation*}
    \bigg|
        \Tr\bigg[
            \prod_{\nu=1}^p
            \bfE_{N,1}(R)^{a_\nu}
            \bfE_{N,2}(R)^{b_\nu}
            \bfM_N(R)^{c_\nu}
            \bfL_N(R)^{d_\nu}
        \bigg]
    \bigg|
    =
    O(e^{-cN})
\end{equation*}
for some $c>0$. Hence, it remains to consider $r\in(0,1]$.

Define the transition window
\begin{equation} \label{def of transition window}
    \mathcal I_n
    :=
    \bigg\{
        j\in\{1,\ldots,n\}:
        |j-nR^2|
        \leq
        (q+1)\sqrt n\log n
    \bigg\}.
\end{equation}
We first record the following localisation estimates. By
Lemmas~\ref{lem: element estimates bfE-n-R-1} and
\ref{lem: element estimates bfE-n-R-2}, together with the quadratic
behaviour of the corresponding rate functions around their minima,
there exists $c>0$ such that
\begin{align}
\begin{cases} \label{estimates for EN12 outside localisation}
[\bfE_{N,1}(R)]_{j,k}
    = O(e^{-c(\log n)^2}),
    & \textup{if } |k-nR^2|>\sqrt n\log n,
\\[3pt]
[\bfE_{N,2}(R)]_{j,k}
    = O(e^{-c(\log n)^2}),
    & \textup{if }
    \max\{|j-nR^2|,|k-nR^2|\}>\sqrt n\log n.
\end{cases}
\end{align}
Moreover, using the bound
$\gamma(a,x)\leq\Gamma(a)$ in the definition of $\bfM_N(R)$,
together with Stirling's approximation as in the proof of
Lemma~\ref{lem: trace power bfM-n-R asymp}, we have
\begin{equation}
    [\bfM_N(R)]_{j,k}
    =
    O(e^{-c(\log n)^2}),
    \qquad
    \textup{if } |j-k|>\sqrt n\log n.
    \label{estimate for MN outside diagonal}
\end{equation}
Here and below, polynomial prefactors are absorbed by decreasing the
constant $c$.

Recall that, by cyclicity of the trace, we have chosen the ordering so
that $\bfQ_q$ is either $\bfE_{N,1}(R)$ or $\bfE_{N,2}(R)$. In
particular, the last factor in the index expansion is $ [\bfQ_q]_{j_q,j_1}.$
By \eqref{estimates for EN12 outside localisation}, its contribution
is exponentially small whenever
\begin{equation*}
    |j_1-nR^2|>\sqrt n\log n.
\end{equation*}
Hence, up to an exponentially small error, we may restrict the
summation to
\begin{equation*}
    |j_1-nR^2|\leq\sqrt n\log n.
\end{equation*} 

We now proceed successively along the ordered product. Suppose that
\begin{equation*}
    |j_\ell-nR^2|
    \leq
    \ell\sqrt n\log n.
\end{equation*}
If $\bfQ_\ell=\bfM_N(R)$ and
\begin{equation*}
    |j_{\ell+1}-nR^2|
    >
    (\ell+1)\sqrt n\log n,
\end{equation*}
then
\begin{equation*}
    |j_{\ell+1}-j_\ell|
    >
    \sqrt n\log n,
\end{equation*}
so the corresponding contribution is exponentially small by \eqref{estimate for MN outside diagonal}. If
$\bfQ_\ell$ is either $\bfE_{N,1}(R)$ or $\bfE_{N,2}(R)$, the same
conclusion follows directly from the corresponding entrywise estimate \eqref{estimates for EN12 outside localisation}.
Iterating this argument shows that, up to an error of order
$O(e^{-c(\log n)^2})$, all indices $j_1,\ldots,j_q$ may be restricted
to the transition window $\mathcal I_n$ in \eqref{def of transition window}.

For $j,k\in\mathcal I_n$, we have $j,k=nR^2+O(\sqrt n\log n)$. Using
$\gamma(a,z)\leq \Gamma(a)$ and the log-convexity of the gamma
function, we obtain
\begin{align*}
0\leq [\bfM_N(R)]_{j,k}
&\leq
\frac{\Gamma(j+k-\frac32)}
{\sqrt{2\pi\Gamma(2j-1)\Gamma(2k-1)}} \leq
\frac{1}{\sqrt{2\pi}}
\frac{\Gamma(j+k-\frac32)}
{\Gamma(j+k-1)}
=
O(n^{-1/2}),
\end{align*}
uniformly for $j,k\in\mathcal I_n$. Similarly,
Lemmas~\ref{lem: element estimates bfE-n-R-1} and
\ref{lem: element estimates bfE-n-R-2} give
\begin{equation*}
    [\bfE_{N,1}(R)]_{j,k}
    =
    O(n^{-1/2}),
    \qquad
    [\bfE_{N,2}(R)]_{j,k}
    =
    O(n^{-1/2}).
\end{equation*}
Note that by definition \eqref{def of transition window}, we have
\begin{equation*}
    |\mathcal I_n|
    =
    O(\sqrt n\log n).
\end{equation*}
Therefore, the contribution from the localised index set is bounded by
\begin{equation*}
    O\big(
        |\mathcal I_n|^q n^{-q/2}
    \big)
    =
    O\big(
        (\log n)^q
    \big)
    =
    o(\sqrt n).
\end{equation*}
Combining this estimate with the exponentially small contribution
from the complement of $\mathcal I_n$ proves the assertion.
\end{proof}

\section{Proof of the main results} \label{sec: proof main rslt}

In this section, we complete the proof of
Theorem~\ref{thm: cumulants} by combining the asymptotic estimates
established in the previous section. Proposition~\ref{prop: number variance}
then follows as the special case corresponding to the second-order
cumulants. Moreover, since all joint cumulants of order at least two
are of the same order, the higher-order cumulants become negligible
under the natural fluctuation scaling. This immediately yields the
bivariate central limit theorem stated in Corollary~\ref{cor: CLT}.

\begin{proof}[Proof of Theorem~\ref{thm: cumulants}]
We first prove the asymptotic formulas \eqref{joint cumulants bulk main thm}, \eqref{joint cumulants edge main thm} and \eqref{joint cumulants outside main thm}. 
Fix nonnegative integers $j$ and $k$ satisfying $j+k\geq2$, and set  
    \begin{equation}
        \mathfrak{C}_{j,k} 
        = \frac{1}{j!k!} \kappa_{j,k}(\cN_N^{\rm r}(R), \cN_N^{\rm c}(R)).
    \end{equation}
By Proposition~\ref{Prop_finite N formula} and
\eqref{eq: trace log expansion}, we have
\begin{equation*}
    \begin{split}
        \mathfrak{C}_{j,k}
        &=
        [t^ju^k]
        \sum_{m=1}^{\infty}
        \frac{(-1)^{m-1}}{m}
        \Tr\biggl[
            \Big(
                (e^{2t}-e^{2u})\bfM_N(R)
                -(e^{2t}-e^t)\bfE_{N,1}(R)
                +(e^{2u}-1)
                \bigl(
                    \bfL_N(R)-\bfE_{N,2}(R)
                \bigr)
            \Big)^m
        \biggr].
    \end{split}
\end{equation*} 
Each scalar coefficient appearing inside the matrix power vanishes at
$(t,u)=(0,0)$. Hence every monomial $t^{\ell_1}u^{\ell_2}$ in the
$m$-th summand satisfies $\ell_1+\ell_2\geq m.$ 
It follows that only the terms with $m\leq j+k$ contribute to the
coefficient of $t^ju^k$. 

We now expand each matrix power into noncommutative words.
Applying
Lemmas~\ref{lem: trace power mixed products only bfM bfL} and
\ref{lem: trace power mixed products including bfE}
to each word, we obtain, as $N=2n\to\infty$,
\begin{equation*}
    \begin{split}
        \mathfrak{C}_{j,k}
        &=
        [t^ju^k]
        \sum_{m=1}^{j+k}
        \frac{(-1)^{m-1}}{m}
        \biggl[
            (e^{2u}-1)^m
            \Tr\bigl[\bfL_N(R)^m\bigr]
        \\
        &\qquad\qquad
            +
            \sum_{\ell=1}^m
            \binom{m}{\ell}
            (e^{2t}-e^{2u})^\ell
            (e^{2u}-1)^{m-\ell}
            \Tr\bigl[\bfM_N(R)^\ell\bigr]
        \biggr]
        +o(\sqrt{n}).
    \end{split}
\end{equation*}
 Indeed, for each $\ell\in\{1,\ldots,m\}$, there are
$\binom{m}{\ell}$ words containing exactly $\ell$ factors of
$\bfM_N(R)$ and $m-\ell$ factors of $\bfL_N(R)$, and the trace of
each such word agrees with
$\Tr[\bfM_N(R)^\ell]$ up to an error of order $o(\sqrt n)$.

Rearranging the finite sums gives
\begin{equation*}
    \begin{split}
        \mathfrak{C}_{j,k}
        &=
        [t^ju^k]
        \bigg(
            \sum_{m=1}^{j+k}
            \frac{(-1)^{m-1}}{m}
            (e^{2u}-1)^m
            \Tr\bigl[\bfL_N(R)^m\bigr]
        \\
        &\qquad
            +
            \sum_{\ell=1}^{j+k}
            (e^{2t}-e^{2u})^\ell
            \Tr\bigl[\bfM_N(R)^\ell\bigr]
            \sum_{m=\ell}^{j+k}
            \frac{(-1)^{m-1}}{m}
            \binom{m}{\ell}
            (e^{2u}-1)^{m-\ell}
        \bigg)
        +o(\sqrt{n}).
    \end{split}
\end{equation*} 
Under the coefficient extraction $[t^ju^k]$, the upper limit of the
innermost summation may be replaced by infinity, since the additional
terms have total degree greater than $j+k$. We then use the identity
\begin{equation*}
    \sum_{m=\ell}^{\infty}
    \frac{(-1)^{m-1}}{m}
    \binom{m}{\ell}
    x^{m-\ell}
    =
    \frac{(-1)^{\ell-1}}{\ell}
    (1+x)^{-\ell}
\end{equation*}
and deduce
    \begin{align} \label{eq: joint cumulants reduced to pure traces}
        \mathfrak{C}_{j,k} = [t^j u^k] \sum_{m = 1}^{j+k} \frac{(-1)^{m-1}}{m} \bigg(   (e^{2u}-1)^m \Tr \big[\bfL_N(R)^m \big] + (e^{2(t-u)}-1)^m \Tr\big[ \bfM_N(R)^m \big]  \bigg) + o(\sqrt{n}).
    \end{align}

Substituting the asymptotic formulas from
Lemmas~\ref{lem: trace power bfL-n-R asymp} and
\ref{lem: trace power bfM-n-R asymp} into
\eqref{eq: joint cumulants reduced to pure traces}, we obtain 
    \begin{align*}
        \mathfrak{C}_{j,k} = [t^j u^k] \sum_{m = 1}^{j+k} \frac{(-1)^{m-1}}{m } \bigg(  (e^{2u}-1)^m \Big( (r^2 \wedge 1)n + \alpha_{m,r,s}\sqrt{n} \Big) + (e^{2(t-u)}-1)^m (r \wedge 1)\sqrt{\frac{n}{\pi m}} \,  \bigg) + o(\sqrt{n}).
    \end{align*}
Since the coefficient of $t^ju^k$ is unchanged if the summations are
extended over all positive integers, by \eqref{def of polylog}, we have 
    \begin{equation} \label{eq: joint cumulants before coefficient extraction}
        \mathfrak{C}_{j,k} = [t^j u^k] \bigg(  2 u (r^2 \wedge 1) n + \sqrt{n} g(u) - (r \wedge 1)\sqrt{\frac{n}{\pi}} \Li_{3/2}(1-e^{2(t-u)})   \bigg) + o(\sqrt{n}),
    \end{equation}
    where
    \begin{equation}
        g(u) := r \int_{-\infty}^{x_{\rm max}} \bigg( \log\Big( 1 + (e^{2u}-1) \frac{\erfc(x)}{2} \Big) - 2u \mathbbm{1}_{\{x<s\}} \bigg) \,dx,
    \end{equation}
    with
    \begin{equation}
        x_{\rm max} := \begin{dcases}
            \infty, & \text{if } 0 < r <1,
            \\
            s, & \text{if } r =1,
            \\
            -\infty, & \text{if } 1 < r \leq \infty.
        \end{dcases}
    \end{equation}

 It remains to extract the relevant coefficients.
We use the expansion
\begin{equation}
    \Li_{3/2}(1-e^z)
    =
    \sum_{q=1}^{\infty}
    \frac{z^q}{q!}
    \sum_{\ell=1}^q
    \frac{(-1)^\ell\ell!}{\ell^{3/2}}
    S(q,\ell),
\end{equation}
whose derivation can be found in the supplementary material to
\cite{PS18}. By \eqref{def of varkappa real}, this gives
\begin{equation}
    \begin{split}
        [t^ju^k]
        \biggl[
            -\Li_{3/2}
            \bigl(
                1-e^{2(t-u)}
            \bigr)
        \biggr]
        &=
        \frac{2(-1)^k}{j!k!}
        \varkappa_{j+k}^{\rm r}.
    \end{split}
    \label{eq: coefficient polylog real cumulants}
\end{equation}

We next compute the coefficients of $g$. For $q\in(0,1)$ and
$k\geq2$, we have
\begin{equation}
    \begin{split}
        [u^k]
        \log\Bigl(
            1+(e^{2u}-1)q
        \Bigr)
        &=
        -\frac{2^k}{k!}
        \Li_{1-k}\biggl(
            -\frac{q}{1-q}
        \biggr).
    \end{split}
\end{equation}
By using $ \erfc(-x)=2-\erfc(x),$  we obtain
\begin{equation}
    \begin{split}
        [u^k]g(u)
        &=
        -\frac{2^kr}{k!}
        \int_{-\infty}^{x_{\rm max}}
        \Li_{1-k}\biggl(
            -\frac{\erfc(x)}{\erfc(-x)}
        \biggr)
        \,dx.
    \end{split}
\end{equation}
For $k\geq2$, the inversion identity
\[
    \Li_{1-k}(z^{-1})
    =
    (-1)^k\Li_{1-k}(z)
\]
and the definition \eqref{def of varkappa total} of $\varkappa_k(s)$ yield
\begin{equation}
    [u^k]g(u)
    =
    \frac{2}{k!\sqrt{\pi}}
    \begin{cases}
        r\varkappa_k(\infty),
        &\textup{if } 0<r<1,
        \smallskip 
        \\
        \varkappa_k(s),
        & \textup{if }r=1,
        \smallskip 
        \\
        0,
        & \textup{if } 1<r\leq\infty.
    \end{cases}
    \label{eq: coefficient of g}
\end{equation}
Moreover, since $g$ depends only on $u$, we have $ [t^ju^k]g(u)=0$ whenever $j\geq1$.
Finally, the linear term $ 2u(r^2\wedge1)n$  does not contribute because $j+k\geq2$. Combining
\eqref{eq: joint cumulants before coefficient extraction},
\eqref{eq: coefficient polylog real cumulants}, and
\eqref{eq: coefficient of g}  we obtain
\eqref{joint cumulants bulk main thm},
\eqref{joint cumulants edge main thm}, and
\eqref{joint cumulants outside main thm}.

    It remains to derive the corresponding asymptotics \eqref{cumulants_total bulk}, \eqref{cumulants_total edge} and \eqref{cumulants_total outside} for the cumulants of
    the total counting function $ \mathcal N_N(R)$.  
    By the relation between ordinary and joint cumulants, we have 
    \begin{equation}\label{eq: cumulant total from joint}
    \kappa_p\bigl(\mathcal N_N(R)\bigr)
    =
    \sum_{j=0}^{p}
    \binom{p}{j}
    \kappa_{j,p-j}
    \bigl(
    \mathcal N_N^{\rm r}(R),
    \mathcal N_N^{\rm c}(R)
    \bigr),
    \qquad p\ge2.
    \end{equation}
    Note that in the bulk regime, by \eqref{joint cumulants bulk main thm} and \eqref{eq: cumulant total from joint}, we obtain 
    \begin{align*}
        \lim_{N\to\infty}
    \frac{\kappa_p\bigl(\mathcal N_N(R)\bigr)}
    {R\sqrt{2N/\pi}}
    &
    =
    \varkappa_p(\infty)
    +
    \varkappa_p^{\rm r}
    \sum_{j=0}^{p}
    \binom{p}{j}(-1)^{p-j} =
    \varkappa_p(\infty),
    \end{align*}
    where the last equality follows from
    \[
    \sum_{j=0}^{p}
    \binom{p}{j}(-1)^{p-j}
    =
    (1-1)^p
    =
    0.
    \]
    Since $\varkappa_p(\infty)=0$ for odd $p$, this yields the stated
    bulk formula.
    The same cancellation occurs in the other regimes. Indeed, for the edge regime, by \eqref{joint cumulants edge main thm} and \eqref{eq: cumulant total from joint}, we have
    \begin{align*}
    \lim_{N\to\infty}
    \frac{\kappa_p\bigl(\mathcal N_N(R)\bigr)}
    {\sqrt{2N/\pi}}= \varkappa_p(s)
    +
    \varkappa_p^{\rm r}
    \sum_{j=0}^{p}
    \binom{p}{j}(-1)^{p-j}= \varkappa_p(s),
    \end{align*}
    and for the outside, by \eqref{joint cumulants outside main thm} and \eqref{eq: cumulant total from joint}, we have
    \begin{align*}
    \lim_{N\to\infty}
    \frac{\kappa_p\bigl(\mathcal N_N(R)\bigr)}
    {\sqrt{2N/\pi}}
    &=
    \varkappa_p^{\rm r}
    \sum_{j=0}^{p}
    \binom{p}{j}(-1)^{p-j}
    =0.
    \end{align*}
    This completes the proof of the assertions \eqref{cumulants_total bulk}, \eqref{cumulants_total edge} and \eqref{cumulants_total outside}. 
\end{proof}

\bibliographystyle{abbrv}

\begin{thebibliography}{100}

 
\bibitem{AB23} G. Akemann and S.-S. Byun, \emph{The product of $m$ real $N \times N$ Ginibre matrices: Real eigenvalues in the critical regime $m=O(N)$}, Constr. Approx. \textbf{59} (2024), 31--59. 

\bibitem{ABE23} G. Akemann, S.-S. Byun and M. Ebke, \emph{Universality of the number variance in rotational invariant two-dimensional Coulomb gases}, J. Stat. Phys. \textbf{190} (2023), 1--34. 

\bibitem{ABES23} G.~Akemann, S.-S. Byun, M. Ebke and G. Schehr, \emph{Universality in the number variance and counting statistics of the real and symplectic Ginibre ensemble}, J. Phys. A \textbf{56} (2023), 495202.

\bibitem{ABL25} G. Akemann, S.-S. Byun and Y.-W. Lee, \emph{The probability of almost all eigenvalues being real for the elliptic real Ginibre ensemble}, Nonlinearity \textbf{38} (2025), 105015.


\bibitem{AIS14} G. Akemann, J. R. Ipsen and E. Strahov, \emph{Permanental processes from products of complex and quaternionic induced Ginibre ensembles}, Random Matrices Theory Appl. \textbf{3} (2014), 1450014.

\bibitem{AK07} G.~Akemann and E.~Kanzieper, \emph{Integrable structure of Ginibre's ensemble of real random matrices and a {P}faffian integration theorem}, J. Stat. Phys. \textbf{129} (2007), 1159--1231. 

\bibitem{AFLS25} M. Allard, P. J. Forrester, S. Lahiry and B.-J. Shen, \emph{Partition function of 2D Coulomb gases with radially symmetric potentials and a hard wall}, arXiv:2506.14738.

\bibitem{ATW14} R.~Allez, J.~Touboul and G.~Wainrib, \emph{Index distribution of the Ginibre ensemble}, J. Phys. A \textbf{47} (2014), 042001.

\bibitem{ACCL23} Y. Ameur, C. Charlier, J. Cronvall and J. Lenells, \emph{Exponential moments for disk counting statistics at the hard edge of random normal matrices}, J. Spectr. Theory {\bf 13} (2023), 841--902.

\bibitem{ACCL24} Y. Ameur, C. Charlier, J. Cronvall and J. Lenells, \emph{Disk counting statistics near hard edges of random normal matrices: the multi-component regime}, Adv. Math. \textbf{441} (2024), 109549.

\bibitem{ACM24} Y. Ameur, C. Charlier and P. Moreillon, \emph{Eigenvalues of truncated unitary matrices: disk counting statistics}, Monatsh. Math. {\bf 204} (2024), 197--216.

\bibitem{BCD26} L. Benigni, S. Coste and G. Dubach, \emph{The delocalization of eigenvectors of real elliptic matrices}, arXiv:2602.10264.

\bibitem{BFK21} G. Ben Arous, Y. V. Fyodorov and B. A. Khoruzhenko, \emph{Counting equilibria of large complex systems by instability index}, Proc. Natl. Acad. Sci. USA \textbf{118} (2021), e2023719118. 

\bibitem{BS09} A. Borodin and C. D. Sinclair, \emph{The Ginibre ensemble of real random matrices and its scaling limits}, Comm. Math. Phys. \textbf{291} (2009), 177--224. 

\bibitem{BC25} S.-S. Byun and C. Charlier, \emph{On the characteristic polynomial of the eigenvalue moduli of random normal matrices}, Constr. Approx. \textbf{62} (2025), 471--521.

\bibitem{BF25} S.-S.~Byun and P. J.~Forrester, \emph{Progress on the study of the Ginibre ensembles}, Springer Singapore, KIAS Springer Series in Mathematics (2025).

\bibitem{BF25a} S.-S.~Byun and P. J.~Forrester, \emph{Electrostatic computations for statistical mechanics and random matrix applications}, 2025 MATRIX Annals Part II, MATRIX Book Series, pp. 211--254. 

\bibitem{BKLL23} S.-S. Byun, N.-G. Kang, J. O. Lee and J. Lee, \emph{Real eigenvalues of elliptic random matrices}, Int. Math. Res. Not. \textbf{2023} (2023), 2243--2280.

\bibitem{BL24} S.-S. Byun and Y.-W. Lee, \emph{Finite size corrections for real eigenvalues of the elliptic Ginibre matrices}, Random Matrices Theory Appl. \textbf{13} (2024), 2450005. 

\bibitem{BLO26} S.-S. Byun, Y.-W. Lee and S. Oh, \emph{Upper tail large deviations for extremal eigenvalues of the real, complex and symplectic elliptic Ginibre matrices}, J. Phys. A \textbf{59} (2026), 305203.

\bibitem{BJLS25} S.-S. Byun, J. Jalowy, Y.-W. Lee and G. Schehr, \emph{Moderate-to-large deviation asymptotics for real eigenvalues of the elliptic Ginibre matrices}, Bernoulli (to appear), arXiv:2511.09191.  

\bibitem{BMS25} S.-S.~Byun, L. D.~Molag and N.~Simm, \emph{Large deviations and fluctuations of real eigenvalues of elliptic random matrices}, Electron. J. Probab. \textbf{30} (2025), 40 pp.

\bibitem{BN25} S.-S. Byun and K. Noda, \emph{Real eigenvalues of asymmetric Wishart matrices: Expected number, global density and integrable structure}, arXiv:2503.14942. 

\bibitem{BP26} S.-S. Byun and S. Park, \emph{Large gap probabilities of complex and symplectic spherical ensembles with point charges}, J. Funct. Anal. \textbf{290} (2026), 111260.
 

\bibitem{Ch22} C. Charlier, \emph{Asymptotics of determinants with a rotation-invariant weight and discontinuities along circles}, Adv. Math. \textbf{408} (2022), 108600.

\bibitem{Ch23} C. Charlier, \emph{Large gap asymptotics on annuli in the random normal matrix model}, Math. Ann. \textbf{388} (2024), 3529--3587.

\bibitem{CL23} C. Charlier and J. Lenells, \emph{Exponential moments for disk counting statistics of random normal matrices in the critical regime}, Nonlinearity {\bf 36} (2023), 1593--1616.

  
\bibitem{CEK26} G. Cipolloni, L. Erd\H{o}s and O. Kolupaiev, \emph{The eigenvalues of i.i.d. matrices are hyperuniform}, arXiv:2602.17628v2.

 
\bibitem{CESX22} G. Cipolloni, L. Erd\H{o}s, D. Schr\"oder and Y. Xu, \emph{Directional extremal statistics for Ginibre eigenvalues}, J. Math. Phys. \textbf{63} (2022), 103303.

\bibitem{CL26} G. Cipolloni and B. Landon, \emph{Gaussian multiplicative chaos for i.i.d.\ matrices}, arXiv:2605.29962.

\bibitem{CMV16} F. D. Cunden, F. Mezzadri and P. Vivo, \emph{Large deviations of radial statistics in the two-dimensional one-component
plasma}, J. Stat. Phys. \textbf{164} (2016), 1062--1081.

\bibitem{Ed97} A. Edelman, \emph{The probability that a random real Gaussian matrix has k real eigenvalues, related distributions, and the circular law}, J. Multivariate Anal. \textbf{60} (1997), 203--232. 

\bibitem{EKS94}
A. Edelman, E. Kostlan and M. Shub, \emph{How many eigenvalues of a random matrix are real?} J. Amer. Math. Soc. \textbf{7} (1994), 247--267.
 

\bibitem{FL22} M. Fenzl and G. Lambert, \emph{Precise deviations for disk counting statistics of invariant determinantal processes}, Int. Math. Res. Not. {\bf 2022} (2022), 7420--7494.


\bibitem{FS23a} W. FitzGerald and N. Simm, \emph{Fluctuations and correlations for products of real asymmetric random matrices}, Ann. Inst. Henri Poincar\'e Probab. Stat. \textbf{59} (2023), 2308--2342.


\bibitem{Fo10} P. J. Forrester, \emph{Log-gases and random matrices}, Princeton University Press, Princeton, NJ (2010).


\bibitem{Fo14} P. J. Forrester, \emph{Probability of all eigenvalues real for products of standard Gaussian matrices}, J. Phys. A \textbf{47} (2014), 065202. 

\bibitem{Fo15a}
P. J. Forrester, \emph{Diffusion processes and the asymptotic bulk gap probability for the real Ginibre ensemble}, J. Phys. A \textbf{48} (2015), 324001.


\bibitem{Fo24} P. J.~Forrester, \emph{Local central limit theorem for real eigenvalue fluctuations of elliptic GinOE matrices}, Electron. Commun. Probab. \textbf{29} (2024), 1--11.

\bibitem{Fo25} P. J.~Forrester, \emph{Asymptotics of the real eigenvalue distribution for the real spherical ensemble}, J. Stat. Phys. \textbf{193} (2026), 56.

\bibitem{FI16} P. J. Forrester and J. R. Ipsen, \emph{Real eigenvalue statistics for products of asymmetric real Gaussian matrices}, Linear Algebra Appl. \textbf{510} (2016), 259--290.

\bibitem{FIK20} P. J. Forrester, J. R. Ipsen and S. Kumar, \emph{How many eigenvalues of a product of truncated orthogonal matrices are real?}, Exp. Math. \textbf{29} (2020), 276--290.

\bibitem{FK18}  P. J. Forrester and S. Kumar, \emph{The probability that all eigenvalues are real for products of truncated real orthogonal random matrices}, J. Theoret. Probab. \textbf{31} (2018), 2056--2071.

\bibitem{FM09} P. J. Forrester and A. Mays, \emph{A method to calculate correlation functions for $\beta = 1$ random matrices of odd size}, J. Stat. Phys. \textbf{134} (2009), 443--462.

\bibitem{FM12} P. J. Forrester and A. Mays, \emph{Pfaffian point processes for the Gaussian real generalised eigenvalue problem}, Prob. Theory and Rel. Fields \textbf{154} (2012), 1--47.


\bibitem{FN07} P. J. Forrester and T. Nagao, \emph{Eigenvalue statistics of the real Ginibre ensemble}, Phys. Rev. Lett. \textbf{99} (2007), 050603.

\bibitem{FN08} P. J. Forrester and T. Nagao, \emph{Skew orthogonal polynomials and the partly symmetric real Ginibre ensemble}, J. Phys. A \textbf{41} (2008), 375003.


\bibitem{Fyo16} Y. V. Fyodorov, \emph{Topology trivialization transition in random nongradient autonomous ODEs on a sphere}, J. Stat. Mech. \textbf{2016} (2016), 124003. 

\bibitem{FK16} Y. V. Fyodorov and B. A. Khoruzhenko, \emph{Nonlinear analogue of the May--Wigner instability transition}, Proc. Natl. Acad. Sci. USA \textbf{113} (2016), 6827--6832.


\bibitem{GP19} M. Gebert and M. Poplavskyi, \emph{On pure complex spectrum for truncations of random orthogonal matrices and Kac polynomials}, arXiv:1905.03154. 

\bibitem{GLX24} A. Goel, P. Lopatto and X. Xie, \emph{Central limit theorem for the complex eigenvalues of Gaussian random matrices}, Electron. Commun. Probab. \textbf{29} (2024), Paper No. 16, 13 pp.

\bibitem{GDMK26} P. Goswami, A. Dhar, S. N. Majumdar and A. Kundu, \emph{Entanglement scaling and full counting statistics in excited states of two-dimensional rotating fermions}, arXiv:2608.05722. 


\bibitem{KA05} E. Kanzieper and G. Akemann, \emph{Statistics of real eigenvalues in Ginibre's ensemble of random real matrices},  Phys. Rev. Lett. {\bf 95} (2005), 230201.


\bibitem{KPTTZ15} E.~Kanzieper, M.~Poplavskyi, C.~Timm, R.~Tribe and O.~Zaboronski, \emph{What is the probability that a large random matrix has no real eigenvalues?}, Ann.~Appl.~Probab. \textbf{26} (2016), 2733--2753.

\bibitem{Kiv24} P. Kivimae, \emph{Concentration of equilibria and relative instability in disordered non-relaxational dynamics}, Comm. Math. Phys. \textbf{405}, 289 (2024). 


\bibitem{LGCCKMS19}
B. Lacroix-A-Chez-Toine, J. A. M. Garz{\'o}n, C. S. H. Calva, I. P. Castillo, A. Kundu, S. N. Majumdar and G. Schehr, \emph{Intermediate deviation regime for the full eigenvalue statistics in the complex Ginibre ensemble},  Phys. Rev. E {\bf 100} (2019), 012137.


\bibitem{LGMS18} B. Lacroix-A-Chez-Toine, A. Grabsch, S. N. Majumdar and G. Schehr, \emph{Extremes of 2d Coulomb gas: universal intermediate deviation regime}, J. Stat. Mech. Theory Exp. {\bf 2018} (2018), 013203.

\bibitem{LMS19}
B. Lacroix-A-Chez-Toine, S. N. Majumdar and G. Schehr, \emph{Rotating trapped fermions in two dimensions and the complex Ginibre ensemble: Exact results for the entanglement entropy and number variance}, Phys. Rev. A \textbf{99} (2019), 021602.

\bibitem{LMS22} A. Little, F. Mezzadri and N. Simm, \emph{On the number of real eigenvalues of a product of truncated orthogonal random matrices}, Electron. J. Probab. \textbf{27} (2022), 32 pp.

\bibitem{LM26} P. Lopatto and M. Meeker, \emph{Smallest gaps between eigenvalues of real Gaussian matrices}, Bernoulli \textbf{32} (2026), 325--349. 

\bibitem{LO25} P. Lopatto and M. Otto, \emph{Maximum gap in complex Ginibre matrices}, arXiv:2501.04611. 

\bibitem{May72} R. M. May, \emph{Will a large complex system be stable?}, Nature {\bf 238} (1972), 413.


\bibitem{MMS14} R. Marino, S.~N. Majumdar and G. Schehr, \emph{Phase transitions and edge scaling of number variance in Gaussian random matrices}, Phys. Rev. Lett. {\bf 112} (2014), 254101.

\bibitem{MMO26} J. Marzo, L.~D. Molag and J. Ortega-Cerd\`a, \emph{Universality for fluctuations of counting statistics of random normal matrices}, J. Lond. Math. Soc. (2) {\bf 113} (2026), 35 pp.


\bibitem{Noda25} K. Noda, \emph{Partition functions of two-dimensional Coulomb gases with circular root- and jump-type singularities}, arXiv:2510.00843. 

 \bibitem{MAD24} L. D. Molag, G. Akemann and M. Duits, \emph{Fluctuations in various regimes of non-Hermiticity and a holographic principle}, Ann. Henri Poincar\'e (Online), arXiv:2412.15854.

\bibitem{MPTW16} L. C. G. del Molino, K. Pakdaman, J. Touboul and G. Wainrib, \emph{The real Ginibre ensemble with $k=O(n)$ real eigenvalues}, J. Stat. Phys. \textbf{163} (2016), 303--323. 


\bibitem{NIST}  F. W. J. Olver, D. W. Lozier, R. F. Boisvert and C. W. Clark, eds. \emph{NIST Handbook of Mathematical Functions}, Cambridge: Cambridge University Press, 2010.


\bibitem{PS18} M. Poplavskyi and G. Schehr, \emph{Exact persistence exponent for the 2d-diffusion equation and related Kac polynomials}, Phys. Rev. Lett. {\bf 121} (2018), 150601, see Supp. Mat. in https://arxiv.org/abs/1806.11275
 

\bibitem{Rider03} B. Rider, \emph{A limit theorem at the edge of a non-Hermitian random matrix ensemble}, J. Phys. A \textbf{36} (2003), 3401--3410. 

\bibitem{RS14} B. Rider and C. D. Sinclair, \emph{Extremal laws for the real Ginibre ensemble}, Ann. Appl. Probab. \textbf{24} (2014), 1621--1651.
 

\bibitem{Si07} C. D. Sinclair, \emph{Averages over {G}inibre's ensemble of random real matrices}, Int. Math. Res. Not. \textbf{2007} (2007), rnm015, 15 pp.

\bibitem{Si17} N. Simm, \emph{Central limit theorems for the real eigenvalues of large Gaussian random matrices}, Random Matrices Theory Appl. \textbf{6} (2017), 1750002. 


\bibitem{Si17a} N. Simm, \emph{On the real spectrum of a product of Gaussian matrices}, Electron. Commun. Probab. \textbf{22} (2017), 11.


\bibitem{SLMS21} N. R. Smith, P. Le Doussal, S. N. Majumdar and G. Schehr, \emph{Full counting statistics for interacting trapped fermions}, SciPost Phys. \textbf{11} (2021), 110.

\bibitem{SLMS22} N. R. Smith, P. Le Doussal, S. N. Majumdar and G. Schehr, \emph{Counting statistics for noninteracting fermions in a rotating trap}, Phys. Rev. A \textbf{105} (2022), 043315. 

\bibitem{So07} H.-J. Sommers, \emph{Symplectic structure of the real Ginibre ensemble}, J. Phys. A \textbf{40} (2007), F671.

\bibitem{SW08} H.-J. Sommers and W.~Wieczorek, \emph{General eigenvalue correlations for the real Ginibre ensemble}, J. Phys. A \textbf{41} (2008), 405003.

\bibitem{Sos02} A. Soshnikov, \emph{Gaussian limit for determinantal random point fields}, Ann. Probab. \textbf{30} (2002), 171--187.

\bibitem{TV15} T. Tao and V. Vu, \emph{Random matrices: Universality of local spectral statistics of non-Hermitian matrices}, Ann. Probab. \textbf{43} (2015), 782--874. 


\bibitem{XZ24} Y. Xu and Q. Zeng, \emph{Large deviations for the extremal eigenvalues of Ginibre ensembles}, Acta Math. Sin. \textbf{42} (2026), 1555--1571.


 \end{thebibliography}

\end{document}